\documentclass[reqno,10pt,centertags,draft]{amsart}
\usepackage{amsmath,amsthm,amscd,amssymb,latexsym,esint,upref,stmaryrd,
enumerate,color,verbatim,yfonts}
\usepackage{hyperref} 
\newcommand*{\mailto}[1]{\href{mailto:#1}{\nolinkurl{#1}}}
\newcommand{\arxiv}[1]{\href{http://arxiv.org/abs/#1}{arXiv:#1}}

\newcommand{\bbC}{{\mathbb{C}}}

\newcommand{\bbN}{{\mathbb{N}}}

\newcommand{\bbR}{{\mathbb{R}}}

\newcommand{\cB}{{\mathcal B}}

\newcommand{\cH}{{\mathcal H}}

\newcommand{\cN}{{\mathcal N}}

\newcommand{\cS}{{\mathcal S}}

\newcommand{\cV}{{\mathcal V}}

\newcommand{\beq}{\begin{equation}}
\newcommand{\enq}{\end{equation}}

\renewcommand{\a}{\alpha}
\renewcommand{\b}{\beta}
\newcommand{\g}{\gamma}
\renewcommand{\d}{\delta}

\DeclareMathOperator{\supp}{supp}

\DeclareMathOperator{\dom}{dom}

\renewcommand{\Re}{\text{\rm Re}}
\renewcommand{\Im}{\text{\rm Im}}
\renewcommand{\ln}{\text{\rm ln}}

\newcommand{\no}{\notag}
\newcommand{\lb}{\label}
\newcommand{\f}{\frac}

\newcommand{\ol}{\overline}
\newcommand{\bs}{\backslash}

\newcommand{\wti}{\widetilde}

\newcommand{\oh}{o}
\newcommand{\hatt}{\widehat} 
\newcommand{\dott}{\,\cdot\,}

\renewcommand{\dot}{\overset{\textbf{\Large.}}}

\newcommand{\dotA}{{\hspace{0.08cm}\dot{\hspace{-0.08cm} A}}}

\newcommand{\bi}{\bibitem}

\let\geq\geqslant
\let\leq\leqslant

\newcommand{\lam}{\lambda}

\newcommand{\al}{\alpha}
\newcommand{\be}{\beta}

\newcommand{\Lr}{{L^2((a,b);rdx)}} 

\newcommand{\ACl}{{AC_{loc}((a,b))}}
\newcommand{\Ll}{{L^1_{loc}((a,b);dx)}}
\newcommand{\SL}{\text{\textnormal{SL}}}

\newcommand{\high}[1]{{\raisebox{1mm}{$#1$}}}

\makeatletter
\def\theequation{\@arabic\c@equation}

\allowdisplaybreaks 
\numberwithin{equation}{section}

\newtheorem{theorem}{Theorem}[section]

\newtheorem{definition}[theorem]{Definition}
\newtheorem{hypothesis}[theorem]{Hypothesis}
\newtheorem{example}[theorem]{Example}

\theoremstyle{remark}
\newtheorem{remark}[theorem]{Remark}

\begin{document}

\title[Donoghue $m$-functions]{Donoghue $m$-Functions for 
Singular Sturm--Liouville Operators}

\author[F.\ Gesztesy]{Fritz Gesztesy}
\address{Department of Mathematics, 
Baylor University, Sid Richardson Bldg., 1410 S.\,4th Street, Waco, TX 76706, USA}
\email{\mailto{Fritz\_Gesztesy@baylor.edu}}
\urladdr{\url{http://www.baylor.edu/math/index.php?id=935340}}

\author[L.\ L.\ Littlejohn]{Lance L. Littlejohn}
\address{Department of Mathematics, 
Baylor University, Sid Richardson Bldg., 1410 S.\,4th Street, Waco, TX 76706, USA}
\email{\mailto{Lance\_Littlejohn@baylor.edu}}
\urladdr{\url{http://www.baylor.edu/math/index.php?id=53980}}

\author[R.\ Nichols]{Roger Nichols}
\address{Department of Mathematics (Dept.~6956), The University of Tennessee at Chattanooga, 
615 McCallie Ave, Chattanooga, TN 37403, USA}
\email{\mailto{Roger-Nichols@utc.edu}}
\urladdr{\url{https://sites.google.com/mocs.utc.edu/rogernicholshomepage/home}}

\author[M. Piorkowski]{Mateusz Piorkowski}
\address{Faculty of Mathematics\\ University of Vienna\\
Oskar-Morgenstern-Platz 1\\ 1090 Wien}
\email{\href{mailto:Mateusz.Piorkowski@univie.ac.at}{Mateusz.Piorkowski@univie.ac.at}}

\author[J.\ Stanfill]{Jonathan Stanfill}
\address{Department of Mathematics, 
Baylor University, Sid Richardson Bldg., 1410 S.\,4th Street, Waco, TX 76706, USA}
\email{\mailto{Jonathan\_Stanfill@baylor.edu}}
\urladdr{\url{http://sites.baylor.edu/jonathan-stanfill/}}

\dedicatory{Dedicated to Sergey Naboko $($1950--2020\,$)$: Friend and Mathematician Extraordinaire}

\date{\today}
\@namedef{subjclassname@2020}{\textup{2020} Mathematics Subject Classification}

\begin{abstract} 

Let $\dotA$ be a densely defined, closed, symmetric operator in the complex, separable Hilbert space $\cH$ with equal deficiency indices and denote by 
$\cN_i = \ker \big(\big(\dotA\big)\high{^*} - i I_{\cH}\big)$, $\dim \, (\cN_i)=k\in \bbN \cup \{\infty\}$, 
the associated deficiency subspace of $\dotA$ . If $A$ denotes a self-adjoint extension of $\dotA$ in $\cH$, 
the Donoghue $m$-operator $M_{A,\cN_i}^{Do} (\dott)$ in $\cN_i$ associated with the pair
 $(A,\cN_i)$  is given by
\[
M_{A,\cN_i}^{Do}(z)=zI_{\cN_i} + (z^2+1) P_{\cN_i} (A - z I_{\cH})^{-1}
P_{\cN_i} \big\vert_{\cN_i}\,, \quad  z\in \bbC\backslash \bbR,     
\]
with $I_{\cN_i}$ the identity operator in $\cN_i$, and $P_{\cN_i}$ the orthogonal projection in
$\cH$ onto $\cN_i$. 

Assuming the standard local integrability hypotheses on the coefficients $p, q,r$, we study all self-adjoint 
realizations corresponding to the differential expression  
\[
\tau=\f{1}{r(x)}\left[-\f{d}{dx}p(x)\f{d}{dx} + q(x)\right] \, \text{ for a.e.~$x\in(a,b) \subseteq \bbR$,}     
\]
in $L^2((a,b); rdx)$, and, as the principal aim of this paper, systematically construct the associated Donoghue $m$-functions (resp., $2 \times 2$ matrices) in all cases where $\tau$ is in the limit circle case at least at one interval endpoint $a$ or $b$.  
\end{abstract}

\maketitle

{\scriptsize{\tableofcontents}}

\maketitle



\section{Introduction} \lb{s1}

{\it Sergey's contributions to operator and spectral theory are legendary and will pass the test of time. He had a very keen eye for the interface of complex analysis and operator theory, and was quite interested in all aspects of $($operator-valued\,$)$ $m$-functions. We hope our modest contribution would have been something he might have enjoyed.}

To set the stage we briefly discuss abstract Donoghue $m$-functions following \cite{GNWZ19} (see also  \cite{GKMT01}, \cite{GMT98}). Given a self-adjoint extension $A$ of a densely defined, closed, symmetric operator 
$\dotA$ in $\cH$ (a complex, separable Hilbert space) with equal deficiency indices and the deficiency subspace $\cN_i$ of $\dotA$ in $\cH$, with
\begin{equation}
 \cN_i = \ker \big(\big(\dotA\big)\high{^*} - i I_{\cH}\big), \quad
\dim \, (\cN_i)=k\in \bbN \cup \{\infty\},     \lb{1.1}
\end{equation}
the Donoghue $m$-operator $M_{A,\cN_i}^{Do} (\dott) \in\cB(\cN_i)$ associated with the pair
 $(A,\cN_i)$  is given by
\begin{align}
\begin{split}
M_{A,\cN_i}^{Do}(z)&=P_{\cN_i} (zA + I_\cH)(A - z I_{\cH})^{-1}
P_{\cN_i}\big\vert_{\cN_i}      \\
&=zI_{\cN_i} + (z^2+1) P_{\cN_i} (A - z I_{\cH})^{-1}
P_{\cN_i} \big\vert_{\cN_i}\,, \quad  z\in \bbC\backslash \bbR,     \lb{1.2}
\end{split}
\end{align}
with $I_{\cN_i}$ the identity operator in $\cN_i$, and $P_{\cN_i}$ the orthogonal projection in
$\cH$ onto $\cN_i$. The special case $k=1$, was discussed in detail by Donoghue \cite{Do65}; for the case $k\in \bbN$ we refer to \cite{GT00}.

More generally, given a self-adjoint extension $A$ of $\dotA$ in $\cH$ and a closed,
linear subspace $\cN$ of $\cH$,
the Donoghue $m$-operator $M_{A,\cN}^{Do} (\dott) \in\cB(\cN)$ associated with the pair
 $(A,\cN)$  is defined by
\begin{align}
\begin{split}
M_{A,\cN}^{Do}(z)&=P_\cN (zA + I_\cH)(A - z I_{\cH})^{-1}
P_\cN\big\vert_\cN      \\
&=zI_\cN+(z^2+1)P_\cN(A - z I_{\cH})^{-1}
P_\cN\big\vert_\cN\,, \quad  z\in \bbC\backslash \bbR,     \lb{1.3}
\end{split}
\end{align}
with $I_\cN$ the identity operator in $\cN$ and $P_\cN$ the orthogonal projection in $\cH$
onto $\cN$. 

Since $M_{A,\cN}^{Do}(z)$ is analytic for $z\in \bbC\backslash\bbR$ and satisfies (see \cite[Theorem~5.3]{GNWZ19}) 
\begin{align}
& [\Im(z)]^{-1} \Im\big(M_{A,\cN}^{Do} (z)\big) \geq
2 \Big[\big(|z|^2 + 1\big) + \big[\big(|z|^2 -1\big)^2 + 4 (\Re(z))^2\big]^{1/2}\Big]^{-1} I_{\cN},    \no \\
& \hspace*{9.5cm}  z\in \bbC\backslash \bbR,    \lb{1.4}
\end{align}
$M_{A,\cN}^{Do}(\dott)$ is a $\cB(\cN)$-valued Nevanlinna--Herglotz function. Thus, $M_{A,\cN}^{Do}(\dott)$ 
admits the representation
\begin{equation}
M_{A,\cN}^{Do}(z) = \int_\bbR
d\Omega_{A,\cN}^{Do}(\lambda) \bigg[\f{1}{\lambda-z} -
\f{\lambda}{\lambda^2 + 1}\bigg], \quad z\in\bbC\backslash\bbR,    \lb{1.5}
\end{equation}
where the $\cB(\cN)$-valued measure $\Omega_{A,\cN}^{Do}(\dott)$ satisfies
\begin{align}
&\Omega_{A,\cN}^{Do}(\lambda)=(\lambda^2 + 1) (P_{\cN} E_A(\lambda)P_{\cN}\big\vert_{\cN}),
\lb{1.6} \\
&\int_\bbR d\Omega_{A,\cN}^{Do}(\lambda) \, (1+\lambda^2)^{-1}=I_{\cN},    \lb{1.7} \\
&\int_\bbR d(\xi,\Omega_{A,\cN}^{Do} (\lambda)\xi)_{\cN} = \infty \, \text{ for all } \,
 \xi \in \cN \backslash \{0\},    \lb{1.8}
\end{align}
with $E_A(\dott)$ the family of strongly right-continuous spectral projections of $A$ in $\cH$. 

Operators of the type $M_{A,\cN}^{Do}(\dott)$ and some of its variants have attracted
considerable attention in the literature. They appear to go back to Krein \cite{Kr46} (see also 
\cite{KL71}), Saakjan \cite{Sa65}, and independently, Donoghue \cite{Do65}. The interested reader can find a wealth of additional information in the context of \eqref{1.2}--\eqref{1.8} in 
 \cite{AB09}, \cite{AP04}, \cite{BL07}--\cite{BR16}, \cite{BMN02}, \cite{BGW09}--\cite{BGP08},
\cite{DHMdS09}--\cite{DMT88}, \cite{GKMT01}--\cite{GT00}, \cite{GWZ13b}, \cite{HMM13}, \cite{KO78}, \cite{LT77}, \cite{Ma92a}, \cite{MN11}, \cite{MN12}, \cite{Ma04}, \cite{Mo09}, \cite{Na89}--\cite{Na94}, \cite{Pa13}, \cite{Po04}, \cite{Ry07}, and the references therein.

Without going into further details (see \cite[Corollary~5.8]{GNWZ19} for details) we note that the prime reason for the interest in $M_{A,\cN_i}^{Do}(\dott)$ lies in the fundamental fact that the entire spectral information of $A$ contained in its family of spectral projections $E_A(\dott)$, is already encoded in the $\cB(\cN_i)$-valued measure $\Omega_{A,\cN_i}^{Do}(\dott)$  (including multiplicity properties of the spectrum of $A$) if and only if $\dotA$ is completely non-self-adjoint in $\cH$ (that is, if and only if $\dotA$ has no invariant subspace on which it is self-adjoint, see 
\cite[Lemma~5.4]{GNWZ19}). 

In the remainder of this paper, we will exclusively focus on the particular case
$\cN = \cN_i = \ker\big(\big(\dotA\big)\high{^*} - i I_{\cH}\big)$ and develop a self-contained approach to constructing Donoghue $m$-functions (resp., $2 \times 2$ matrices) for singular Sturm--Liouville operators on arbitrary intervals 
$(a,b) \subseteq \bbR$. More precisely, assuming the standard local integrability hypotheses on the coefficients $p, q,r$ (cf.\ Hypothesis \ref{h2.1}) we study all self-adjoint $L^2((a,b); rdx)$-realizations corresponding to the differential expression  
\begin{equation}
\tau=\f{1}{r(x)}\left[-\f{d}{dx}p(x)\f{d}{dx} + q(x)\right] \, \text{ for a.e.~$x\in(a,b) \subseteq \bbR$,}     \lb{1.9}
\end{equation} 
and systematically determine the underlying Donoghue $m$-functions in all cases where $\tau$ is in the limit circle case at least at one interval endpoint $a$ or $b$.  

Turning to the content of each section, we discuss the necessary background in connection with minimal $T_{min}$ and maximal $T_{max}$ operators, self-adjoint extensions, etc., corresponding to \eqref{1.9} in the underlying Hilbert space $L^2((a,b); rdx)$ in Section \ref{s2}. In particular, we recall the discussion of boundary values in terms of appropriate Wronskians, especially, in the case where $T_{min}$ is bounded from below (utilizing principal and nonprincipal solutions). Our strategy for the construction of Donoghue $m$-functions consists of first constructing them for the Friedrichs extension of $T_{min}$ and then employing Krein-type resolvent formulas to derive Donoghue $m$-functions for the remaining self-adjoint extensions of $T_{min}$. These Krein-type resolvent formulas use the Friedrichs extension as a reference operator and then explicitly characterize the resolvents of all the remaining self-adjoint extensions of $T_{min}$ in terms of the Friedrichs extension and the deficiency subspaces for $T_{min}$. Hence Sections \ref{s3} and \ref{s4} derive Krein-type resolvent formulas for singular Sturm--Liouville operators in the case where $\tau$ has one, respectively, two, interval endpoints in the limit circle case. Donoghue $m$-functions corresponding to the case where $\tau$ is in the limit circle case in precisely one interval endpoint are derived in Section \ref{s5}; the case where 
$\tau$ is in the limit circle case at $a$ and $b$ is treated in detail in Section \ref{s6}. We conclude this paper with an illustration of a generalized Bessel operator in Section \ref{s7} where $a=0$, $b\in(0,\infty)\cup\{\infty\}$, and $\tau$ takes on the explicit form, 
\begin{align}
\begin{split}
\tau_{\d,\nu,\g} = x^{-\d}\left[-\frac{d}{dx}x^\nu\frac{d}{dx} +\frac{(2+\d-\nu)^2\g^2-(1-\nu)^2}{4}x^{\nu-2}\right],\\
\d>-1,\; \nu<1,\; \g\geq0,\; x\in(0,b).     \lb{1.10} 
\end{split}
\end{align}

Finally, we comment on some of the basic notation used throughout this paper.  If $T$ is a linear operator mapping (a subspace of) a Hilbert space into another, then $\dom(T)$ and $\ker(T)$ denote the domain and kernel (i.e., null space) of $T$. The spectrum and resolvent set of a closed linear operator in a Hilbert space will be denoted by $\sigma(\dott)$ and $\rho(\dott)$, respectively. Moreover, we typically abbreviate $L^2((a,b); r dx)$ as $L^2_r((a,b))$ in various subscripts involving the identity operator $I_{L^2_r((a,b))}$ and the scalar product $(\dott,\dott)_{L^2_r((a,b))}$ (linear in the second argument) and associated norm $\|\dott\|_{L^2_r((a,b))}$ in $L^2((a,b); r dx)$.

\section{Some Background} \lb{s2}

In this section we briefly recall the basics of singular Sturm--Liouville operators. The material is standard and can be found, for instance, in \cite[Ch.~6]{BHS20}, \cite[Chs.~8, 9]{CL85}, \cite[Sects.~13.6, 13.9, 13.10]{DS88}, \cite{EGNT13}, \cite[Ch.~4]{GZ21}, \cite[Ch.~III]{JR76}, \cite[Ch.~V]{Na68}, \cite{NZ92}, \cite[Ch.~6]{Pe88}, \cite[Ch.~9]{Te14}, \cite[Sect.~8.3]{We80}, \cite[Ch.~13]{We03}, \cite[Chs.~4, 6--8]{Ze05}.

Throughout this section we make the following assumptions:

\begin{hypothesis} \lb{h2.1}
Let $(a,b) \subseteq \bbR$ and suppose that $p,q,r$ are $($Lebesgue\,$)$ measurable functions on $(a,b)$ 
such that the following items $(i)$--$(iii)$ hold: \\[1mm] 
$(i)$ \hspace*{1.1mm} $r>0$ a.e.~on $(a,b)$, $r\in\Ll$. \\[1mm] 
$(ii)$ \hspace*{.1mm} $p>0$ a.e.~on $(a,b)$, $1/p \in\Ll$. \\[1mm] 
$(iii)$ $q$ is real-valued a.e.~on $(a,b)$, $q\in\Ll$. 
\end{hypothesis}

Given Hypothesis \ref{h2.1}, we study Sturm--Liouville operators associated with the general, 
three-coefficient differential expression $\tau$ of the form
\begin{equation}
\tau=\f{1}{r(x)}\left[-\f{d}{dx}p(x)\f{d}{dx} + q(x)\right] \, \text{ for a.e.~$x\in(a,b) \subseteq \bbR$.}     \lb{2.1}
\end{equation} 

If $f\in AC_{loc}((a,b))$, then the quasi-derivative of $f$ is defined to be $f^{[1]}:=pf'$.  Moreover, the Wronskian of two functions $f,g\in AC_{loc}((a,b))$ is defined by
\begin{equation}
W(f,g)(x) = f(x)g^{[1]}(x)-f^{[1]}(x)g(x)\,\text{ for a.e.~$x\in (a,b)$}.
\end{equation}

The following result is useful for computing weighted integrals of products of solutions of $(\tau - z)y=0$: Assume Hypothesis \ref{h2.1} and let $z_1,z_2\in \bbC$ with $z_1\neq z_2$.  If $y(z_j,\dott)$ is a solution of $(\tau - z_j)y=0$, $j\in \{1,2\}$, then for all $a<\alpha<\beta<b$,
\begin{equation}
\int_{\alpha}^{\beta} r(x)dx\, y(z_1,x)y(z_2,x) = \frac{W(y(z_1,\dott),y(z_2,\dott))\big|_{\alpha}^{\beta}}{z_1-z_2}.     \lb{2.1b} 
\end{equation}

\begin{definition} \lb{d2.2}
Assume Hypothesis \ref{h2.1}. Given $\tau$ as in \eqref{2.1}, the \textit{maximal operator} $T_{max}$ in $\Lr$ associated with $\tau$ is defined by
\begin{align}
&T_{max} f = \tau f,    \no
\\
& f \in \dom(T_{max})=\big\{g\in\Lr \, \big| \,g,g^{[1]}\in\ACl;   \lb{2.2} \\ 
& \hspace*{6.3cm}  \tau g\in\Lr\big\}.   \no
\end{align}
The \textit{preminimal operator} $\dot T_{min} $ in $\Lr$ associated with $\tau$ is defined by 
\begin{align}
&\dot T_{min}  f = \tau f,   \no
\\
&f \in \dom \big(\dot T_{min}\big)=\big\{g\in\Lr \, \big| \, g,g^{[1]}\in\ACl;   \lb{2.3}
\\
&\hspace*{3.25cm} \supp \, (g)\subset(a,b) \text{ is compact; } \tau g\in\Lr\big\}.   \no
\end{align}

One can prove that $\dot T_{min} $ is closable, and one then defines the \textit{minimal operator} $T_{min}$ as the closure of $\dot T_{min} $.
\end{definition}

For $f,g\in \dom(T_{max})$, one can prove that the following limits exist:
\begin{equation}
W(f,g)(a) = \lim_{x\downarrow a} W(f,g)(x)\quad \text{and}\quad W(f,g)(b) = \lim_{x\uparrow b} W(f,g)(x).
\end{equation}
In addition, one can prove the following basic fact:

\begin{theorem} \lb{t2.3} 
Assume Hypothesis \ref{h2.1}. Then 
\begin{equation} 
\big(\dot T_{min}\big)^* = T_{max}, 
\end{equation} 
and hence $T_{max}$ is closed and $T_{min}=\ol{\dot T_{min} }$ is given by
\begin{align}
&T_{min} f = \tau f, \no
\\
&f \in \dom(T_{min})=\big\{g\in\Lr  \, \big| \,  g,g^{[1]}\in\ACl;       \lb{2.8} \\
& \qquad \text{for all } h\in\dom(T_{max}), \, W(h,g)(a)=0=W(h,g)(b); \, \tau g\in\Lr\big\}     \no \\
& \quad =\big\{g\in\dom(T_{max})  \, \big| \, W(h,g)(a)=0=W(h,g)(b) \, 
\text{for all } h\in\dom(T_{max}) \big\}.   \no 
\end{align}
Moreover, $\dot T_{min} $ is essentially self-adjoint if and only if\; $T_{max}$ is symmetric, and then $\ol{\dot T_{min} }=T_{min}=T_{max}$.
\end{theorem}

Regarding self-adjoint extensions of $T_{min}$ one has the following first result.

\begin{theorem} \lb{t2.4} 
Assume Hypothesis \ref{h2.1}.
An extension $\wti T$ of $\dot T_{min} $ or of $T_{min}=\ol{\dot T_{min} }$ is self-adjoint if and only if 
\begin{align}
& \wti T f = \tau f,    \\
& f \in \dom\big(\wti T\big) = \big\{g\in\dom(T_{max})  \, \big| \,  
W(f,g)(a)=W(f,g)(b) \text{ for all } f\in\dom\big(\wti T\big)\big\}.   \no 
\end{align}
\end{theorem}

The celebrated Weyl alternative then can be stated as follows:

\begin{theorem}[Weyl's Alternative] \lb{t2.5} ${}$ \\
Assume Hypothesis \ref{h2.1}. Then the following alternative holds: Either \\[1mm] 
$(i)$ for every $z\in\bbC$, all solutions $u$ of $(\tau-z)u=0$ are in $\Lr$ near $b$ 
$($resp., near $a$$)$, \\[1mm] 
or, \\[1mm] 
$(ii)$ for every $z\in\bbC$, there exists at least one solution $u$ of $(\tau-z)u=0$ which is not in $\Lr$ near $b$ $($resp., near $a$$)$. In this case, for each $z\in\bbC\bs\bbR$, there exists precisely one solution $u_b$ $($resp., $u_a$$)$ of $(\tau-z)u=0$ $($up to constant multiples$)$ which lies in $\Lr$ near $b$ $($resp., near $a$$)$. 
\end{theorem}

This yields the limit circle/limit point classification of $\tau$ at an interval endpoint as follows. 

\begin{definition} \lb{d2.6} 
Assume Hypothesis \ref{h2.1}. \\[1mm]  
In case $(i)$ in Theorem \ref{t2.5}, $\tau$ is said to be in the \textit{limit circle case} at $b$ $($resp., $a$$)$. $($Frequently, $\tau$ is then called \textit{quasi-regular} at $b$ $($resp., $a$$)$.$)$
\\[1mm] 
In case $(ii)$ in Theorem \ref{t2.5}, $\tau$ is said to be in the \textit{limit point case} at $b$ $($resp., $a$$)$. \\[1mm]
If $\tau$ is in the limit circle case at $a$ and $b$ then $\tau$ is also called \textit{quasi-regular} on $(a,b)$. 
\end{definition}

The next result links self-adjointness of $T_{min}$ (resp., $T_{max}$) and the limit point property of $\tau$ at both endpoints.  Here, and throughout, we shall employ the notation
\begin{equation}
\cN_z = \ker\big(T_{max}-zI_{L_r^2((a,b))}\big),\quad z\in \bbC.
\end{equation}

\begin{theorem} \lb{t2.7}
Assume Hypothesis~\ref{h2.1}, then the following items $(i)$ and $(ii)$ hold: \\[1mm] 
$(i)$ If $\tau$ is in the limit point case at $a$ $($resp., $b$$)$, then 
\begin{equation} 
W(f,g)(a)=0 \, \text{$($resp., $W(f,g)(b)=0$$)$ for all $f,g\in\dom(T_{max})$.} 
\end{equation} 
$(ii)$ Let $T_{min}=\ol{\dot T_{min} }$. Then
\begin{align}
\begin{split}
n_\pm(T_{min}) &= \dim\big(\cN_{\pm i}\big)    \\
& = \begin{cases}
2 & \text{if $\tau$ is in the limit circle case at $a$ and $b$,}\\
1 & \text{if $\tau$ is in the limit circle case at $a$} \\
& \text{and in the limit point case at $b$, or vice versa,}\\
0 & \text{if $\tau$ is in the limit point case at $a$ and $b$}.
\end{cases}
\end{split} 
\end{align}
In particular, $T_{min} = T_{max}$ is self-adjoint if and only if $\tau$ is in the limit point case at $a$ and $b$. 
\end{theorem}

All self-adjoint extensions of $T_{min}$ are then described as follows:

\begin{theorem} \lb{t2.8}
Assume Hypothesis \ref{h2.1} and that $\tau$ is in the limit circle case at $a$ and $b$ $($i.e., $\tau$ is quasi-regular 
on $(a,b)$$)$. In addition, assume that 
$v_j \in \dom(T_{max})$, $j=1,2$, satisfy 
\begin{equation}
W(\ol{v_1}, v_2)(a) = W(\ol{v_1}, v_2)(b) = 1, \quad W(\ol{v_j}, v_j)(a) = W(\ol{v_j}, v_j)(b) = 0, \; j= 1,2.  
\end{equation}
$($E.g., real-valued solutions $v_j$, $j=1,2$, of $(\tau - \lambda) u = 0$ with $\lambda \in \bbR$, such that 
$W(v_1,v_2) = 1$.$)$ For $g\in\dom(T_{max})$ we introduce the generalized boundary values 
\begin{align}
\begin{split} 
\wti g_1(a) &= - W(v_2, g)(a), \quad \wti g_1(b) = - W(v_2, g)(b),    \\
\wti g_2(a) &= W(v_1, g)(a), \quad \;\,\,\, \wti g_2(b) = W(v_1, g)(b).   \lb{2.10}
\end{split} 
\end{align}
Then the following items $(i)$--$(iii)$ hold: \\[1mm] 
$(i)$ All self-adjoint extensions $T_{\al,\be}$ of $T_{min}$ with separated boundary conditions are of the form
\begin{align}
& T_{\al,\be} f = \tau f, \quad \al,\be\in[0,\pi),  \lb{2.11} \\
&f\in \dom(T_{\alpha,\beta}) = \bigg\{g\in \dom(T_{max})\,\bigg|\,
\begin{aligned}
\cos(\alpha)\wti g_1(a)+\sin(\alpha)\wti g_2(a)&=0,\no\\
\cos(\beta)\wti g_1(b)+\sin(\beta)\wti g_2(b)&=0
\end{aligned}
\bigg\}.\no 
\end{align}
$(ii)$ All self-adjoint extensions $T_{\varphi,R}$ of $T_{min}$ with coupled boundary conditions are of the type
\begin{align}
\begin{split} 
& T_{\varphi,R} f = \tau f,    \\
& f \in \dom(T_{\varphi,R})=\bigg\{g\in\dom(T_{max}) \, \bigg| \begin{pmatrix} \wti g_1(b)\\ \wti g_2(b)\end{pmatrix} 
= e^{i\varphi}R \begin{pmatrix}
\wti g_1(a)\\ \wti g_2(a)\end{pmatrix} \bigg\}, \lb{2.12}
\end{split}
\end{align}
where $\varphi\in[0,2\pi)$, and $R$ is a real $2\times2$ matrix with $\det(R)=1$ 
$($i.e., $R \in SL(2,\bbR)$$)$.  \\[1mm] 
$(iii)$ Every self-adjoint extension of $T_{min}$ is either of type $(i)$ $($i.e., separated\,$)$ or of type 
$(ii)$ $($i.e., coupled\,$)$.
\end{theorem}

\begin{remark} \lb{r2.9}
$(i)$ If $\tau$ is in the limit point case at one endpoint, say, at the endpoint $b$, one omits the corresponding boundary condition involving $\beta \in [0, \pi)$ at $b$ in \eqref{2.11} to obtain all self-adjoint extensions $T_{\alpha}$ of 
$T_{min}$, indexed by $\alpha \in [0, \pi)$. (In this case item $(iii)$ in Theorem \ref{t2.8} is vacuous.) In the case where $\tau$ is in the limit point case at both endpoints, all boundary values and boundary conditions become superfluous as in this case $T_{min} = T_{max}$ is self-adjoint. \\[1mm] 
$(ii)$ In the special case where $\tau$ is regular on the finite interval $[a,b]$, choose $v_j \in \dom(T_{max})$, $j=1,2$, 
such that 
\begin{align}
v_1(x) = \begin{cases} \theta_0(\lambda,x,a), & \text{for $x$ near $a$}, \\
\theta_0(\lambda,x,b), & \text{for $x$ near $b$},  \end{cases}   \quad 
v_2(x) = \begin{cases} \phi_0(\lambda,x,a), & \text{for $x$ near $a$}, \\
\phi_0(\lambda,x,b), & \text{for $x$ near $b$},  \end{cases}   \lb{2.13}
\end{align} 
where $\phi_0(\lambda,\, \cdot \,,d)$, $\theta_0(\lambda,\, \cdot \,,d)$, $d \in \{a,b\}$, are real-valued solutions of $(\tau - \lambda) u = 0$, $\lambda \in \bbR$, satisfying the boundary conditions 
\begin{align}
\begin{split} 
& \phi_0(\lambda,a,a) = \theta_0^{[1]}(\lambda,a,a) = 0, \quad \theta_0(\lambda,a,a) = \phi_0^{[1]}(\lambda,a,a) = 1, \\ 
& \phi_0(\lambda,b,b) = \theta_0^{[1]}(\lambda,b,b) = 0, \quad \; \theta_0(\lambda,b,b) = \phi_0^{[1]}(\lambda,b,b) = 1. 
\lb{2.14}
\end{split} 
\end{align} 
Then one verifies that
\begin{align}
\wti g_1 (a) = g(a), \quad \wti g_1 (b) = g(b), \quad \wti g_2 (a) = g^{[1]}(a), \quad \wti g_2 (b) = g^{[1]}(b),   \lb{2.15} 
\end{align}
and hence Theorem \ref{t2.8} recovers the well-known special regular case. \\[1mm]
$(iii)$ In connection with \eqref{2.10}, an explicit calculation demonstrates that for 
$g, h \in \dom(T_{max})$,
\begin{equation}
\wti g_1(d) \wti h_2(d) - \wti g_2(d) \wti h_1(d) = W(g,h)(d), \quad d \in \{a,b\},   \lb{2.16}
\end{equation} 
interpreted in the sense that either side in \eqref{2.16} has a finite limit as $d \downarrow a$ and $d \uparrow b$. 
Of course, for \eqref{2.16} to hold at $d \in \{a,b\}$, it suffices that $g$ and $h$ lie locally in $\dom(T_{max})$ near $x=d$. \\[1mm]
$(iv)$ Clearly, $\wti g_1, \wti g_2$ depend on the choice of $v_j$, $j=1,2$, and a more precise notation would indicate this as $\wti g_{1,v_2}, \wti g_{2,v_1}$, etc. \\[1mm]
$(v)$ One can supplement the characterization \eqref{2.8} of $\dom(T_{min})$ by
\begin{align}
\begin{split} 
&T_{min} f = \tau f,      \lb{2.17} \\
&f \in \dom(T_{min}) = \big\{g\in\dom(T_{max})  \, \big| \, \wti g_1 (a) = \wti g_2 (a) = \wti g_1 (b) = \wti g_2 (b) = 0\big\}.   \end{split} 
\end{align} 
\hfill $\diamond$
\end{remark} 

Next, we determine when two self-adjoint extensions of $T_{min}$ are {\it relatively prime} with respect to $T_{min}$.

\begin{definition}
If $T$ and $T'$ are self-adjoint extensions of a symmetric operator $S$, then the {\it maximal common part} of $T$ and $T'$ is the operator $C_{T,T'}$ defined by
\begin{equation}
C_{T,T'}u = Tu,\quad u\in \dom(C_{T,T'})=\{f\in \dom(T)\cap\dom(T')\,|\, Tf=T'f\}.
\end{equation}
Moreover, $T$ and $T'$ are said to be relatively prime with respect to $S$ if $C_{T,T'}=S$.
\end{definition}

\begin{theorem}
Assume Hypothesis \ref{h2.1}.\\[1mm]
$(i)$ If $\alpha, \alpha', \beta, \beta'\in [0,\pi)$ with $\alpha\neq \alpha'$ and $\beta\neq \beta'$, then $T_{\alpha,\beta}$ and $T_{\alpha',\beta'}$ are relatively prime with respect to $T_{min}$.\\[1mm]
$(ii)$  If $\alpha,\beta,\beta'\in [0,\pi)$ with $\beta\neq \beta'$, then the maximal common part of $T_{\alpha,\beta}$ and $T_{\alpha,\beta'}$ is the restriction of $T_{max}$ to the subspace
\begin{equation}
\big\{g\in \dom(T_{max})\,\big|\, \cos(\alpha)\wti g_1(a) + \sin(\alpha)\wti g_2(a)=0,\, \wti g_1(b)=\wti g_2(b)=0\big\}.
\end{equation}
$(iii)$  If $\alpha,\alpha',\beta\in [0,\pi)$ with $\alpha\neq \alpha'$, then the maximal common part of $T_{\alpha,\beta}$ and $T_{\alpha',\beta}$ is the restriction of $T_{max}$ to the subspace
\begin{equation}\lb{2.24y}
\big\{g\in \dom(T_{max})\,\big|\, \wti g_1(a)=\wti g_2(a)=0,\, \cos(\beta)\wti g_1(b) + \sin(\beta)\wti g_2(b)=0\big\}.
\end{equation}
$(iv)$  Let $\alpha,\beta\in [0,\pi)$, $\varphi\in [0,2\pi)$, $R=(R_{j,k})_{j,k=1}^2\in \SL(2,\bbR)$, and define
\begin{equation}
\begin{split}
d(\alpha,\beta,R) &= \cos(\alpha)\cos(\beta)R_{1,2}+\cos(\alpha)\sin(\beta)R_{2,2}\\
&\quad -\sin(\alpha)\cos(\beta)R_{1,1}-\sin(\alpha)\sin(\beta)R_{2,1}.
\end{split}
\end{equation}
If $d(\alpha,\beta,R)\neq 0$, then $T_{\alpha,\beta}$ and $T_{\varphi,R}$ are relatively prime with respect to $T_{min}$.  If $d(\alpha,\beta,R)=0$, then the maximal common part of $T_{\alpha,\beta}$ and $T_{\varphi,R}$ is the restriction of $T_{max}$ to the subspace
\begin{equation}\lb{2.26y}
\big\{g\in \dom(T_{\varphi,R})\,\big|\, \cos(\alpha)\wti g_1(a)+\sin(\alpha)\wti g_2(a)=0\big\}.
\end{equation}
$(v)$ Let $\varphi,\eta\in[0,2\pi)$ and $R,S\in \SL(2,\bbR)$.  If $\det\big(e^{i(\eta-\varphi)}SR^{-1}-I_{\bbC^2}\big)\neq 0$, then $T_{\varphi,R}$ and $T_{\eta,S}$ are relatively prime with respect to $T_{min}$.  If 
$\det\big(e^{i(\eta-\varphi)}SR^{-1}-I_{\bbC^2}\big)=0$, so that $1$ is an eigenvalue of $e^{i(\eta-\varphi)}SR^{-1}$ with corresponding eigenspace $\cV_1\subset\bbC^2$, then the maximal common part of $T_{\varphi,R}$ and $T_{\eta,S}$ is the restriction of $T_{max}$ to the subspace
\begin{equation}\lb{2.27y}
\big\{g\in \dom(T_{\varphi,R})\,\big|\, (\wti g_1(b), \wti g_2(b))^{\top}\in \cV_1\big\}.
\end{equation}
\end{theorem}
\begin{proof}
To prove $(i)$, it suffices to show that $f\in \dom(T_{\alpha,\beta})\cap\dom(T_{\alpha',\beta'})$ implies $f\in \dom(T_{min})$.  If $f\in \dom(T_{\alpha,\beta})\cap\dom(T_{\alpha',\beta'})$, then
\begin{align}
\begin{pmatrix}
\cos(\alpha) & \sin(\alpha)\\
\cos(\alpha') & \sin(\alpha')
\end{pmatrix}
\begin{pmatrix}
\wti f_1(a)\\[1mm]
\wti f_2(a)
\end{pmatrix}&=
\begin{pmatrix}
0\\
0
\end{pmatrix},\lb{2.23yy}\\
\begin{pmatrix}
\cos(\beta) & \sin(\beta)\\
\cos(\beta') & \sin(\beta')
\end{pmatrix}
\begin{pmatrix}
\wti f_1(b)\\[1mm]
\wti f_2(b)
\end{pmatrix}&=
\begin{pmatrix}
0\\
0
\end{pmatrix}.\lb{2.23y}
\end{align}
The determinants of the $2\times 2$ coefficient matrices in \eqref{2.23yy} and \eqref{2.23y} are $\sin(\alpha-\alpha')$ and $\sin(\beta-\beta')$, respectively.  Since the assumptions on $\alpha$, $\alpha'$, $\beta$, and $\beta'$ imply $\alpha-\alpha',\beta-\beta'\in (-\pi,\pi)\backslash\{0\}$, it follows that the coefficient matrices in \eqref{2.23yy} and \eqref{2.23y} are invertible.  Hence, $\wti f_1(a)=\wti f_2(a)=\wti f_1(b)=\wti f_2(b)=0$, and the characterization of $\dom(T_{min})$ in \eqref{2.17} implies $f\in \dom(T_{min})$.

The proofs of $(ii)$ and $(iii)$ are similar, so we only provide the proof of $(ii)$ here.  Let $D$ denote the set in \eqref{2.24y}.  To prove $(ii)$, it suffices to show $\dom(T_{\alpha,\beta})\cap\dom(T_{\alpha,\beta'})=D$.  If $f\in \dom(T_{\alpha,\beta})\cap\dom(T_{\alpha,\beta'})$, then $\cos(\alpha)\wti f_1(a)+\sin(\alpha)\wti f_2(a)=0$ and \eqref{2.23y} holds.  As in the proof of $(i)$, the determinant of the $2\times 2$ coefficient matrix in \eqref{2.23y} is nonzero.  Therefore, $\wti f_1(b)=\wti f_2(b)=0$, and it follows that $f\in D$.  Conversely, if $f\in D$, then it is clear that $f$ simultaneously belongs to $\dom(T_{\alpha,\beta})$ and $\dom(T_{\alpha,\beta'})$.

The proof of $(iv)$ begins with a general observation about functions in the intersection $\dom(T_{\alpha,\beta})\cap\dom(T_{\varphi,R})$.  If $f\in\dom(T_{\alpha,\beta})\cap\dom(T_{\varphi,R})$, then
\begin{equation}\lb{2.29y}
\begin{split}
\cos(\alpha)\wti f_1(a) + \sin(\alpha)\wti f_2(a)&=0,\\
\cos(\beta)\wti f_1(b) + \sin(\beta)\wti f_2(b)&=0,
\end{split}
\end{equation}
and
\begin{equation}\lb{2.30y}
\begin{split}
\wti f_1(b)&=e^{i\varphi}R_{1,1}\wti f_1(a) + e^{i\varphi}R_{1,2}\wti f_2(a),\\
\wti f_2(b)&= e^{i\varphi}R_{2,1}\wti f_1(a) + e^{i\varphi}R_{2,2}\wti f_2(a).
\end{split}
\end{equation}
Applying \eqref{2.30y} in \eqref{2.29y} yields a set of boundary conditions that may be recast in matrix form as
\begin{equation}\lb{2.31y}
\begin{pmatrix}
\cos(\alpha) & \sin(\alpha)\\
\cos(\beta)R_{1,1}+\sin(\beta)R_{2,1} & \cos(\beta)R_{1,2}+\sin(\beta)R_{2,2}
\end{pmatrix}
\begin{pmatrix}
\wti f_1(a)\\[1mm]
\wti f_2(a)
\end{pmatrix}
=
\begin{pmatrix}
0\\
0
\end{pmatrix}.
\end{equation}
The determinant of the $2\times 2$ coefficient matrix in \eqref{2.31y} is $d(\alpha,\beta,R)$.

If $d(\alpha,\beta,R)\neq 0$, then \eqref{2.31y} implies $\wti f_1(a)=\wti f_2(a)=0$.  In turn, \eqref{2.30y} implies $\wti f_1(b)=\wti f_2(b)=0$.  Hence, $f\in \dom(T_{min})$, and it follows that $T_{\alpha,\beta}$ and $T_{\varphi,R}$ are relatively prime with respect to $T_{min}$.

To complete the proof of $(iv)$, it remains to show that the set in \eqref{2.26y}, call it $D$, coincides with $\dom(T_{\alpha,\beta})\cap\dom(T_{\varphi,R})$ when $d(\alpha,\beta,R)=0$.  The containment $\dom(T_{\alpha,\beta})\cap\dom(T_{\varphi,R})\subset D$ follows immediately from the definitions of $T_{\alpha,\beta}$, $T_{\varphi,R}$, and $D$.  To prove the reverse containment, let $f\in D$, so that $f\in \dom(H_{\varphi,R})$ and $f$ satisfies the boundary condition at $a$ in \eqref{2.29y}.  The proof is then reduced to showing $f$ satisfies the boundary condition at $b$ in \eqref{2.29y}.  In order to do this, one distinguishes the cases $\alpha\neq 0$ and $\alpha=0$.  If $\alpha\neq 0$, one uses $d(\alpha,\beta,R)=0$, the conditions in \eqref{2.30y}, and $\sin(\alpha)\wti f_2(a)=-\cos(\alpha)\wti f_1(a)$ to compute
\begin{align}
&e^{-i\varphi}\sin(\alpha) \big[\cos(\beta)\wti f_1(b) + \sin(\beta)\wti f_2(b)\big]\lb{2.32y}\\
&\quad = [\cos(\beta)R_{1,1}+\sin(\beta)R_{2,1}]\sin(\alpha)\wti f_1(a)\no\\
&\qquad -[\cos(\beta)R_{1,2}+\sin(\beta)R_{2,2}]\cos(\alpha)\wti f_1(a)\no\\
&\quad = -d(\alpha,\beta,R)\wti f_1(a)\no\\
&\quad= 0.\no
\end{align}
Since $e^{-i\varphi}\sin(\alpha)\neq 0$ when $\alpha\neq 0$, \eqref{2.32y} implies $f$ satisfies the boundary condition at $b$ in \eqref{2.29y}.  If $\alpha=0$, then $\wti f_1(a)=0$, and \eqref{2.30y} simplifies.  One then computes
\begin{align}
\cos(\beta) \wti f_1(b) + \sin(\beta)\wti f_2(b) &= e^{i\varphi}[\cos(\beta)R_{1,2}+\sin(\beta)R_{2,2}]\wti f_2(a)\\
&=e^{i\varphi}d(0,\beta,R)\wti f_2(a)\no\\
&=0,\no
\end{align}
so $f$ satisfies the boundary condition at $b$ in \eqref{2.29y}.

To prove $(v)$, let $f\in \dom(T_{\varphi,R})\cap\dom(T_{\eta,S})$, so that
\begin{equation}\lb{2.35y}
\begin{pmatrix}
\wti f_1(b)\\[1mm]
\wti f_2(b)
\end{pmatrix}=
e^{i\eta}S
\begin{pmatrix}
\wti f_1(a)\\[1mm]
\wti f_2(a)
\end{pmatrix}\quad \text{and}\quad
\begin{pmatrix}
\wti f_1(b)\\[1mm]
\wti f_2(b)
\end{pmatrix}=
e^{i\varphi}R
\begin{pmatrix}
\wti f_1(a)\\[1mm]
\wti f_2(a)
\end{pmatrix}.
\end{equation}
Using the invertibility of $e^{i\varphi}R$ to solve the second equation in \eqref{2.35y} for the vector 
$\big(\wti f_1(a), \wti f_2(a)\big)^{\top}$ and substituting into the first equation in \eqref{2.35y} yields
\begin{equation}\lb{2.36y}
\big[e^{i(\eta-\varphi)}SR^{-1}-I_{\bbC^2}\big]
\begin{pmatrix}
\wti f_1(b)\\[1mm]
\wti f_2(b)
\end{pmatrix}=
\begin{pmatrix}
0\\
0
\end{pmatrix}.
\end{equation}
If $\det\big(e^{i(\eta-\varphi)}SR^{-1}-I_{\bbC^2}\big)\neq 0$, then \eqref{2.36y} implies $\wti f_1(b)=\wti f_2(b)=0$.  In turn, the invertibility of $e^{i\varphi}R$ and the second equation in \eqref{2.35y} yields $\wti f_1(a)=\wti f_2(a)=0$.  Hence, $f\in \dom(T_{min})$, and it follows that $T_{\varphi,R}$ and $T_{\eta,S}$ are relatively prime with respect to $T_{min}$.

Now, suppose that $\det\big(e^{i(\eta-\varphi)}SR^{-1}-I_{\bbC^2}\big)=0$, so that $1$ is an eigenvalue of $e^{i(\eta-\varphi)}SR^{-1}$ with corresponding eigenspace $\cV_1$.  Let $D$ denote the subspace in \eqref{2.27y}.  To complete the proof of $(v)$, it suffices to show the subspace $D$ coincides with $\dom(T_{\varphi,R})\cap\dom(T_{\eta,S})$.  To this end, let $f\in \dom(T_{\varphi,R})\cap\dom(T_{\eta,S})$, so that both equalities in \eqref{2.35y} hold.  In particular, \eqref{2.36y} holds due to the invertibility of $e^{i\varphi}R$, and one concludes that 
$\big(\wti f_1(b), \wti f_2(b)\big)^{\top}\in \cV_1$.  Therefore, $f\in D$.  Conversely, if $f\in D$, then $f\in \dom(T_{\varphi,R})$, and one only needs to show $f\in \dom(T_{\eta,S})$ to complete the proof.  Using the boundary conditions implied by the inclusion $f\in \dom(T_{\varphi,R})$ (i.e., the second equality in \eqref{2.35y}), one computes
\begin{equation}\lb{2.37y}
e^{i\eta}S
\begin{pmatrix}
\wti f_1(a)\\[1mm]
\wti f_2(a)
\end{pmatrix}
=
e^{i(\eta-\varphi)}SR^{-1}
\begin{pmatrix}
\wti f_1(b)\\[1mm]
\wti f_2(b)
\end{pmatrix}=
\begin{pmatrix}
\wti f_1(b)\\[1mm]
\wti f_2(b)
\end{pmatrix},
\end{equation}
where the last equality in \eqref{2.37y} follows from the fact that $\big(\wti f_1(b),\wti f_2(b)\big)^{\top}\in \cV_1$ by the assumption $f\in D$.  The equality in \eqref{2.37y} implies $f\in \dom(T_{\eta,S})$.
\end{proof}

Finally, we turn to the characterization of generalized boundary values in the case where $T_{min}$ is bounded from below following \cite{GLN20} and \cite{NZ92}.

We recall the basics of oscillation theory with particular emphasis on principal and nonprincipal solutions, a notion originally due to Leighton and Morse \cite{LM36} (see also Rellich \cite{Re43}, \cite{Re51} and Hartman and Wintner \cite[Appendix]{HW55}). Our outline below follows \cite{CGN16}, 
\cite[Sects.~13.6, 13.9, 13.10]{DS88}, \cite[Ch.~7]{GZ21}, \cite[Ch.~XI]{Ha02}, \cite{NZ92}, \cite[Chs.~4, 6--8]{Ze05}. 

\begin{definition} \lb{d2.10}
Assume Hypothesis \ref{h2.1}. \\[1mm] 
$(i)$ Fix $c\in (a,b)$ and $\lambda\in\bbR$. Then $\tau - \lam$ is
called {\it nonoscillatory} at $a$ $($resp., $b$$)$, 
if every real-valued solution $u(\lambda,\dott)$ of 
$\tau u = \lambda u$ has finitely many
zeros in $(a,c)$ $($resp., $(c,b)$$)$. Otherwise, $\tau - \lam$ is called {\it oscillatory}
at $a$ $($resp., $b$$)$. \\[1mm] 
$(ii)$ Let $\lambda_0 \in \bbR$. Then $T_{min}$ is called bounded from below by $\lambda_0$, 
and one writes $T_{min} \geq \lambda_0 I_{L_r^2((a,b))}$, if 
\begin{equation} 
\big(u, [T_{min} - \lambda_0 I_{L_r^2((a,b))}]u\big)_{L^2((a,b);rdx)}\geq 0, \quad u \in \dom(T_{min}).
\end{equation}
\end{definition}

The following is a key result. 

\begin{theorem} \lb{t2.11} 
Assume Hypothesis \ref{h2.1}. Then the following items $(i)$--$(iii)$ are
equivalent: \\[1mm] 
$(i)$ $T_{min}$ $($and hence any symmetric extension of $T_{min})$
is bounded from below. \\[1mm] 
$(ii)$ There exists a $\nu_0\in\bbR$ such that for all $\lambda < \nu_0$, $\tau - \lam$ is
nonoscillatory at $a$ and $b$. \\[1mm]
$(iii)$ For fixed $c, d \in (a,b)$, $c \leq d$, there exists a $\nu_0\in\bbR$ such that for all
$\lambda<\nu_0$, $\tau u = \lambda u$ has $($real-valued\,$)$ nonvanishing solutions
$u_a(\lambda,\dott) \neq 0$,
$\hatt u_a(\lambda,\dott) \neq 0$ in the neighborhood $(a,c]$ of $a$, and $($real-valued\,$)$ nonvanishing solutions
$u_b(\lambda,\dott) \neq 0$, $\hatt u_b(\lambda,\dott) \neq 0$ in the neighborhood $[d,b)$ of
$b$, such that 
\begin{align}
&W(\hatt u_a (\lambda,\dott),u_a (\lambda,\dott)) = 1,
\quad u_a (\lambda,x)=\oh(\hatt u_a (\lambda,x))
\text{ as $x\downarrow a$,} \lb{2.18} \\
&W(\hatt u_b (\lambda,\dott),u_b (\lambda,\dott))\, = 1,
\quad u_b (\lambda,x)\,=\oh(\hatt u_b (\lambda,x))
\text{ as $x\uparrow b$,} \lb{2.19} \\
&\int_a^c dx \, p(x)^{-1}u_a(\lambda,x)^{-2}=\int_d^b dx \, 
p(x)^{-1}u_b(\lambda,x)^{-2}=\infty,  \lb{2.20} \\
&\int_a^c dx \, p(x)^{-1}{\hatt u_a(\lambda,x)}^{-2}<\infty, \quad 
\int_d^b dx \, p(x)^{-1}{\hatt u_b(\lambda,x)}^{-2}<\infty. \lb{2.21}
\end{align}
\end{theorem}

\begin{definition} \lb{d2.12}
Assume Hypothesis \ref{h2.1}, suppose that $T_{min}$ is bounded from below, and let 
$\lambda\in\bbR$. Then $u_a(\lambda,\dott)$ $($resp., $u_b(\lambda,\dott)$$)$ in Theorem
\ref{t2.11}\,$(iii)$ is called a {\it principal} $($or {\it minimal}\,$)$
solution of $\tau u=\lambda u$ at $a$ $($resp., $b$$)$. A real-valued solution 
$\wti{\wti u}_a(\lambda,\dott)$ $($resp., $\wti{\wti u}_b(\lambda,\dott)$$)$ of $\tau
u=\lambda u$ linearly independent of $u_a(\lambda,\dott)$ $($resp.,
$u_b(\lambda,\dott)$$)$ is called {\it nonprincipal} at $a$ $($resp., $b$$)$. In particular, $\hatt u_a (\lambda,\dott)$ 
$($resp., $\hatt u_b(\lambda,\dott)$$)$ in \eqref{2.18}--\eqref{2.21} are nonprincipal solutions at $a$ $($resp., $b$$)$.  
\end{definition}

Next, we revisit  in Theorem \ref{t2.8} how the generalized boundary values are utilized in the description of all self-adjoint extensions of $T_{min}$ in the case where $T_{min}$ is bounded from below. 

\begin{theorem} [{\cite[Theorem~4.5]{GLN20}}]\lb{t2.13}
Assume Hypothesis \ref{h2.1} and that $\tau$ is in the limit circle case at $a$ and $b$ $($i.e., $\tau$ is quasi-regular 
on $(a,b)$$)$. In addition, assume that $T_{min} \geq \lambda_0 I$ for some $\lambda_0 \in \bbR$, and denote by 
$u_a(\lambda_0, \dott)$ and $\hatt u_a(\lambda_0, \dott)$ $($resp., $u_b(\lambda_0, \dott)$ and 
$\hatt u_b(\lambda_0, \dott)$$)$ principal and nonprincipal solutions of $\tau u = \lambda_0 u$ at $a$ 
$($resp., $b$$)$, satisfying
\begin{equation}
W(\hatt u_a(\lambda_0,\dott), u_a(\lambda_0,\dott)) = W(\hatt u_b(\lambda_0,\dott), u_b(\lambda_0,\dott)) = 1.  
\lb{2.22} 
\end{equation}
Introducing $v_j \in \dom(T_{max})$, $j=1,2$, via 
\begin{align}
v_1(x) = \begin{cases} \hatt u_a(\lambda_0,x), & \text{for $x$ near a}, \\
\hatt u_b(\lambda_0,x), & \text{for $x$ near b},  \end{cases}   \quad 
v_2(x) = \begin{cases} u_a(\lambda_0,x), & \text{for $x$ near a}, \\
u_b(\lambda_0,x), & \text{for $x$ near b},  \end{cases}   \lb{2.23}
\end{align} 
one obtains for all $g \in \dom(T_{max})$, 
\begin{align}
\begin{split} 
\wti g(a) &= - W(v_2, g)(a) = \wti g_1(a) =  - W(u_a(\lambda_0,\dott), g)(a) = \lim_{x \downarrow a} \f{g(x)}{\hatt u_a(\lambda_0,x)},    \\
\wti g(b) &= - W(v_2, g)(b) = \wti g_1(b) =  - W(u_b(\lambda_0,\dott), g)(b) 
= \lim_{x \uparrow b} \f{g(x)}{\hatt u_b(\lambda_0,x)},    \lb{2.24} 
\end{split} \\
\begin{split} 
{\wti g}^{\, \prime}(a) &= W(v_1, g)(a) = \wti g_2(a) = W(\hatt u_a(\lambda_0,\dott), g)(a) 
= \lim_{x \downarrow a} \f{g(x) - \wti g(a) \hatt u_a(\lambda_0,x)}{u_a(\lambda_0,x)},    \\ 
{\wti g}^{\, \prime}(b) &= W(v_1, g)(b) = \wti g_2(b) = W(\hatt u_b(\lambda_0,\dott), g)(b)  
= \lim_{x \uparrow b} \f{g(x) - \wti g(b) \hatt u_b(\lambda_0,x)}{u_b(\lambda_0,x)}.    \lb{2.25}
\end{split} 
\end{align}
In particular, the limits on the right-hand sides in \eqref{2.24}, \eqref{2.25} exist. 
\end{theorem}

\begin{remark} \lb{r2.14}
The notion of ``generalized boundary values'' in \eqref{2.10} and \eqref{2.24}, \eqref{2.25} corresponds to ``boundary values for $\tau$'' in the sense of \cite[p.~1297, 1304--1307]{DS88}, see also 
\cite[Sect.~3]{Fu73}, \cite[p.~57]{Fu77}. \hfill $\diamond$
\end{remark}

The Friedrichs extension $T_F$ of $T_{min}$ now permits a particularly simple characterization in terms of the generalized boundary values $\wti g(a), \wti g(b)$ as derived by Niessen and Zettl \cite{NZ92}(see also \cite{GP79}, \cite{Ka72}, \cite{Ka78}, \cite{KKZ86}, \cite{MZ00}, \cite{Re51}, \cite{Ro85}, \cite{YSZ15}):

\begin{theorem} \lb{t2.15}
Assume Hypothesis \ref{h2.1} and that $\tau$ is in the limit circle case at $a$ and $b$ $($i.e., $\tau$ is quasi-regular 
on $(a,b)$$)$. In addition, assume that $T_{min} \geq \lambda_0 I$ for some $\lambda_0 \in \bbR$. Then the Friedrichs extension $T_F=T_{0,0}$ of $T_{min}$ is characterized by
\begin{align}
T_F f = \tau f, \quad f \in \dom(T_F)= \big\{g\in\dom(T_{max})  \, \big| \, \wti g(a) = \wti g(b) = 0\big\}.    \lb{2.26}
\end{align}
\end{theorem}

\begin{remark} \lb{r2.16}
$(i)$ As in \eqref{2.16}, one readily verifies for $g, h \in \dom(T_{max})$,
\begin{equation}
\wti g(d) {\wti h}^{\, \prime}(d) - {\wti g}^{\, \prime}(d) \wti h(d) = W(g,h)(d), \quad d \in \{a,b\},    \lb{2.27} 
\end{equation} 
again interpreted in the sense that either side in \eqref{2.27} has a finite limit as $d \downarrow a$ and 
$d \uparrow b$. \\[1mm] 
$(ii)$ As always in this context (cf.\ Remark \ref{r2.9}\,$(i)$), if $\tau$ is in the limit point case at one (or both) interval endpoints, the corresponding boundary conditions at that endpoint are dropped in Theorems \ref{t2.13} and \ref{t2.15}.  
${}$ \hfill $\diamond$
\end{remark}

\section{Krein Resolvent Identities: One Limit Circle Endpoint} \lb{s3}

Assuming that $\tau$ is in the limit circle case at $a$ and in the limit point case at $b$, we derive in this section the Krein resolvent formulas for all self-adjoint extensions of $T_{min}$ using the Friedrichs extension as the reference operator. 

\begin{hypothesis}\lb{h3.1}
In addition to Hypothesis \ref{h2.1} assume that $\tau$ is in the limit circle case at $a$ and in the limit point case at $b$. 
Moreover, for $z\in \rho(T_0)$, let $\psi(z,\dott)$ denote the unique solution to $(\tau -z)y=0$ that satisfies $\psi(z,\dott)\in L_r^2((a,b))$ and $\wti \psi(z,a)=1$.
\end{hypothesis}

Assume Hypothesis \ref{h3.1}. By Theorem \ref{t2.8} or Theorem \ref{t2.13}, the following statements $(i)$ and $(ii)$ hold.\\[1mm]
$(i)$  If $\alpha\in [0,\pi)$, then the operator $T_{\alpha}$ defined by
\begin{equation}\lb{3.1}
\begin{split}
&T_{\alpha}f=T_{max}f,\\
&f\in \dom(T_{\alpha})=\{g\in \dom(T_{max})\,|\, \cos(\alpha)\wti g(a) + \sin(\alpha)\wti g^{\, \prime}(a)=0\}
\end{split}
\end{equation}
is a self-adjoint extension of $T_{min}$.\\[1mm]
$(ii)$  If $T$ is a self-adjoint extension of $T_{min}$, then $T=T_{\alpha}$ for some $\alpha\in[0,\pi)$.\\[2mm]
Statements analogous to $(i)$ and $(ii)$ hold if $\tau$ is in the limit point case at $a$ and in the limit circle case at $b$; for brevity we omit the details.

Choosing $\alpha=0$ in \eqref{3.1} yields the self-adjoint extension $T_0$ with a Dirichlet-type boundary condition at $a$:
\begin{equation}\lb{3.2}
\dom(T_0)=\{g\in\dom(T_{max})\, |\, \wti g(a) = 0\}.
\end{equation}

Since the coefficients $p$, $q$, and $r$ are real-valued, the solution $\psi(z,\dott)$ has the following conjugation property:
\begin{equation}\lb{3.3}
\overline{\psi(z,\dott)} = \psi(\overline{z},\dott),\quad z\in \rho(T_0).
\end{equation}

\begin{theorem}\lb{t3.3}
Assume Hypothesis \ref{h3.1}.  If $\alpha\in (0,\pi)$, then $T_0$ and $T_{\alpha}$ are relatively prime with respect to $T_{\min}$.  Moreover, for each $z\in \rho(T_0)\cap\rho(T_{\alpha})$, the scalar
\begin{equation}\lb{3.4a}
k_{\alpha}(z) = -\cot(\alpha) - \wti \psi^{\, \prime}(z,a)
\end{equation}
is nonzero and
\begin{equation}\lb{3.6}
(T_{\alpha}-zI_{L^2_r((a,b))})^{-1} = (T_0-zI_{L^2_r((a,b))})^{-1} + k_{\alpha}(z)^{-1} (\psi(\overline{z},\dott),\dott)_{L^2_r((a,b))}\psi(z,\dott).
\end{equation}
\end{theorem}
\begin{proof}
The claims follow as a direct application of \cite[Theorem 3.4]{AKMNR19} which is stated in terms of boundary conditions bases and the Lagrange bracket.  The condition
\begin{equation}
W(\hatt u_a(\lambda_0,\dott), u_a(\lambda_0,\dott)) = 1
\end{equation}
implies
\begin{equation}\lb{3.8}
\text{$\{u_a(\lambda_0,\dott),\hatt u_a(\lambda_0,\dott)\}$ is a boundary condition basis at $x=a$}
\end{equation}
in the sense of \cite[Definition 2.15]{AKMNR19} and \cite[Definition 10.4.3]{Ze05}.  The generalized boundary values take the form
\begin{equation}\lb{3.9}
[g,u_a(\lambda_0,\dott)](a) = \wti g(a),\quad [g,\hatt u_a(\lambda_0,\dott)](a) = -{\wti g}^{\, \prime}(a),\\
\end{equation}
where $[\dott,\dott]$ denotes the Lagrange bracket:
\begin{equation}\lb{3.9b}
[f,g](x) = f(x)\overline{(pg')(x)} - (pf')(x)\overline{g(x)},\quad x\in (a,b).
\end{equation}
Using the boundary condition basis in \eqref{3.8} and the identities in \eqref{3.9}, the claims now follow from \cite[Theorem 3.4]{AKMNR19} after a standard reparametrizion of the self-adjoint extensions \eqref{3.1} to fit the parametrization used in \cite[Theorem 2.19]{AKMNR19}.
\end{proof}

\section{Krein Resolvent Identities: Two Limit Circle Endpoints} \lb{s4}

Assuming that $\tau$ is in the limit circle case at $a$ and $b$, we now derive the Krein resolvent formulas for all self-adjoint extensions of $T_{min}$ using once more the Friedrichs extension as the reference operator (in this context we also refer to \cite{CGNZ14}). 

\begin{hypothesis}\lb{h4.1}
In addition to Hypothesis \ref{h2.1} assume that $\tau$ is in the limit circle case at $a$ and $b$. Moreover, for $z\in \rho(T_{0,0})$, let $\{u_j(z,\dott)\}_{j=1,2}$ denote solutions to $\tau u = zu$ which satisfy the boundary conditions
\begin{equation}\lb{4.52}
\begin{split}
\wti u_1(z,a)=0, &\quad \wti u_1(z,b)=1,\\
\wti u_2(z,a)=1, &\quad \wti u_2(z,b)=0.
\end{split}
\end{equation}
\end{hypothesis}

Assume Hypotheses \ref{h4.1}. By Theorem \ref{t2.8} or Theorem \ref{t2.13}, the following statements $(i)$--$(iii)$ hold.\\[1mm]
$(i)$  If $\alpha,\beta\in [0,\pi)$, then the operator $T_{\alpha,\beta}$ defined by
\begin{align}
&T_{\alpha,\beta}f = T_{max}f,\lb{4.49}\\
&f\in \dom(T_{\alpha,\beta}) = \bigg\{g\in \dom(T_{max})\,\bigg|\,
\begin{aligned}
\cos(\alpha)\wti g(a)+\sin(\alpha)\wti g^{\, \prime}(a)&=0,\no\\
\cos(\beta)\wti g(b)+\sin(\beta)\wti g^{\, \prime}(b)&=0
\end{aligned}
\bigg\},\no
\end{align}
is a self-adjoint extension of $T_{min}$.\\[1mm]
$(ii)$  If $\varphi\in [0,2\pi)$ and $R\in \SL(2,\bbR)$, then the operator $T_{\varphi,R}$ defined by
\begin{align}
&T_{\varphi,R}f = T_{max}f,\lb{4.50}\\
&f\in \dom(T_{\varphi,R}) = \bigg\{g\in \dom(T_{max})\,\bigg|\,
\begin{pmatrix}
\wti g(b)\\
\wti g^{\, \prime}(b)
\end{pmatrix}
= e^{i\varphi}R
\begin{pmatrix}
\wti g(a)\\
\wti g^{\, \prime}(a)
\end{pmatrix}
\bigg\},\no
\end{align}
is a self-adjoint extension of $T_{min}$.\\[1mm]
$(iii)$  If $T$ is a self-adjoint extension of $T_{min}$, then $T=T_{\alpha,\beta}$ for some $\alpha,\beta\in[0,\pi)$ or $T=T_{\varphi,R}$ for some $\varphi\in [0,2\pi)$ and some $R\in \SL(2,\bbR)$. \\[2mm] 
\noindent 
{\bf Notational Convention.} {\it To describe all possible self-adjoint boundary conditions associated with self-adjoint extensions of $T_{min}$ effectively, we will frequently employ the notation $T_{A,B}$, $M_{A,B}^{Do}(\dott)$, etc., where $A,B$ represents $\alpha,\beta$ in the case of separated boundary conditions and $\varphi,R$ in the context of coupled boundary conditions.}

\smallskip

Choosing $\alpha=\beta=0$ in \eqref{4.49} yields the self-adjoint extension with Dirichlet-type boundary conditions at $a$ and $b$:
\begin{equation}\lb{4.51}
\dom(T_{0,0})=\{g\in\dom(T_{max})\, |\, \wti g(a) = \wti g(b)=0\}.
\end{equation}

Since the coefficients of the Sturm--Liouville differential expression are real, the following conjugation property holds:
\begin{equation}\lb{4.5}
\ol{u_j(z,\dott)} = u_j(\ol{z},\dott),\quad z\in \rho(T_{0,0}),\, j\in\{1,2\}.
\end{equation}
Applying \eqref{4.52}, one computes
\begin{equation}\lb{4.6}
\begin{split}
W(u_1(z,\dott),u_2(z,\dott)(a) &= -\wti u_1^{\, \prime}(z,a),\\
W(u_1(z,\dott),u_2(z,\dott)(b) &= \wti u_2^{\, \prime}(z,b),\quad z\in \rho(T_{0,0}).
\end{split}
\end{equation}
In particular, since the Wronskian of two solutions is constant,
\begin{equation}\lb{4.7}
\wti u_2^{\, \prime}(z,b) = -\wti u_1^{\, \prime}(z,a),\quad z\in \rho(T_{0,0}).
\end{equation}

\begin{theorem}\lb{t4.3}
Assume Hypothesis \ref{h4.1}. Then the following statements $(i)$--$(v)$ hold.\\[1mm]
$(i)$ If $\alpha,\beta\in(0,\pi)$, then $T_{0,0}$ and $T_{\alpha,\beta}$ are relatively prime with respect to $T_{min}$.  Moreover, for each $z\in \rho(T_{0,0})\cap\rho(T_{\alpha,\beta})$ the matrix
\begin{equation}\lb{4.53}
K_{\alpha,\beta}(z) =
\begin{pmatrix}
\cot(\beta)+\wti u_1^{\, \prime}(z,b) & -\wti u_1^{\, \prime}(z,a)\\[2mm]
\wti u_2^{\, \prime}(z,b) & -\cot(\alpha)-\wti u_2^{\, \prime}(z,a)
\end{pmatrix}
\end{equation}
is invertible and
\begin{align}
(T_{\alpha,\beta}-zI_{L^2_r((a,b))})^{-1} &= (T_{0,0}-zI_{L^2_r((a,b))})^{-1}\no\\
&\quad + \sum_{j,k=1}^2\big[K_{\alpha,\beta}(z)^{-1}\big]_{j,k} ( u_j(\ol{z},\dott),\,\cdot\,)_{L^2_r((a,b))} u_k(z,\dott).\lb{4.54}
\end{align}
$(ii)$  If $\beta\in (0,\pi)$, then the maximal common part of $T_{0,0}$ and $T_{0,\beta}$ is the restriction of $T_{max}$ to the set
\begin{equation}\lb{4.55}
\cS_1=\{y\in \dom(T_{max})\,|\, \wti y(a) = \wti y(b) = \wti y^{\, \prime}(b)=0\}.
\end{equation}
Moreover, for each $z\in \rho(T_{0,0})\cap\rho(T_{0,\beta})$ the scalar
\begin{equation}\lb{4.56}
K_{0,\beta}(z)=\cot(\beta)+\wti u_1^{\, \prime}(z,b)
\end{equation}
is nonzero and
\begin{align}
&(T_{0,\beta}-zI_{L^2_r((a,b))})^{-1}\lb{4.57}\\
&\quad= (T_{0,0}-zI_{L^2_r((a,b))})^{-1} + K_{0,\beta}(z)^{-1} ( u_1(\ol{z},\dott),\,\cdot\,)_{L^2_r((a,b))} u_1(z,\dott).\no
\end{align}
$(iii)$  If $\alpha\in (0,\pi)$, then the maximal common part of $T_{0,0}$ and $T_{\alpha,0}$ is the restriction of $T_{max}$ to the set
\begin{equation}\lb{4.58}
\cS_2=\{y\in \dom(T_{max})\,|\, \wti y(a)=\wti y(b)=\wti y^{\, \prime}(a)=0\}.
\end{equation}
Moreover, for each $z\in \rho(T_{0,0})\cap\rho(T_{\alpha,0})$ the scalar
\begin{equation}\lb{4.59}
K_{\alpha,0}(z)=-\cot(\alpha)-\wti u_2^{\, \prime}(z,a)
\end{equation}
is nonzero and
\begin{align}
&(T_{\alpha,0}-zI_{L^2_r((a,b))})^{-1}\lb{4.60}\\
&\quad = (T_{0,0}-zI_{L^2_r((a,b))})^{-1} + K_{\alpha,0}(z)^{-1} ( u_2(\ol{z},\dott),\,\cdot\,)_{L^2_r((a,b))} u_2(z,\dott).\no
\end{align}
$(iv)$ If $R_{1,2}\neq 0$, then $T_{0,0}$ and $T_{\varphi,R}$ are relatively prime with respect to $T_{min}$.  Moreover, for each $z\in \rho(T_{0,0})\cap\rho(T_{\varphi,R})$ the matrix
\begin{equation}\lb{4.61}
K_{\varphi,R}(z) = \begin{pmatrix}
-\dfrac{R_{2,2}}{R_{1,2}}+\wti u_1^{\, \prime}(z,b) & \dfrac{e^{-i\varphi}}{R_{1,2}}-\wti u_1^{\, \prime}(z,a)\\[4mm]
\dfrac{e^{i\varphi}}{R_{1,2}}+\wti u_2^{\, \prime}(z,b) & -\dfrac{R_{1,1}}{R_{1,2}}-\wti u_2^{\, \prime}(z,a)
\end{pmatrix}
\end{equation}
is invertible and
\begin{align}
(T_{\varphi,R}-zI_{L^2_r((a,b))})^{-1} &= (T_{0,0}-zI_{L^2_r((a,b))})^{-1}\lb{4.62}\\
&\quad + \sum_{j,k=1}^2\big[K_{\varphi,R}(z)^{-1}\big]_{j,k} ( u_j(\ol{z},\dott),\,\cdot\,)_{L^2_r((a,b))} u_k(z,\dott).\no
\end{align}
$(v)$ If $R_{1,2}=0$, then the maximal common part of $T_{\varphi,R}$ and $T_{0,0}$ is the restriction of $T_{max}$ to the set
\begin{align}\lb{4.63}
\cS_{\varphi,R}=\{y\in\dom(T_{max})\,|\,\wti y(a)=\wti y(b)=0,\,\wti y^{\, \prime}(b)=e^{i\varphi}R_{2,2}\wti y^{\, \prime}(a)\}.
\end{align}
Moreover, for each $z\in\rho(T_{0,0})\cap\rho(T_{\varphi,R})$, the scalar
\begin{align}\lb{4.64}
k_{\varphi,R}(z)=-R_{2,1}R_{2,2}-e^{i\varphi}R_{2,2}\wti u_{\varphi,R}^{\, \prime}(z,a)+\wti u_{\varphi,R}^{\, \prime}(z,b)
\end{align}
is nonzero, and
\begin{equation}\lb{4.65}
\begin{split}
(T_{\varphi,R}-zI_{L^2_r((a,b))})^{-1}&=(T_{0,0}-zI_{L^2_r((a,b))})^{-1}\\
&\quad+k_{\varphi,R}(z)^{-1} (u_{\varphi,R}(\ol{z},\dott),\dott)_{L^2_r((a,b))}u_{\varphi,R}(z,\dott),
\end{split}
\end{equation}
where
\begin{align}\lb{4.66}
u_{\varphi,R}(\zeta,\dott) = e^{-i\varphi}R_{2,2}u_2(\zeta,\dott)+u_1(\zeta,\dott),\quad \zeta\in \rho(T_{0,0}).
\end{align}
\end{theorem}
\begin{proof}
Statements $(i)$--$(v)$ are direct applications of the Krein identities for singular Sturm--Liouville operators obtained in \cite{AKMNR19} which are stated in terms of boundary conditions bases and the Lagrange bracket.  The conditions
\begin{equation}
W(\hatt u_a(\lambda_0,\dott), u_a(\lambda_0,\dott)) = W(\hatt u_b(\lambda_0,\dott), u_b(\lambda_0,\dott)) = 1
\lb{4.67}
\end{equation}
imply that
\begin{equation}\lb{4.68}
\text{$\{u_c(\lambda_0,\dott),\hatt u_c(\lambda_0,\dott)\}$ is a boundary condition basis at $x=c$ for $c\in \{a,b\}$}
\end{equation}
in the sense of \cite[Definition 2.15]{AKMNR19} and \cite[Definition 10.4.3]{Ze05}.  The generalized boundary values take the form
\begin{equation}\lb{4.69}
\begin{split}
&[g,u_a(\lambda_0,\dott)](a) = \wti g(a),\quad [g,u_b(\lambda_0,\dott)](b) = \wti g(b),\\
&[g,\hatt u_a(\lambda_0,\dott)](a) = -{\wti g}^{\, \prime}(a),\quad [g,\hatt u_b(\lambda_0,\dott)](b) = -{\wti g}^{\, \prime}(b),
\end{split}
\end{equation}
where $[\dott,\dott]$ denotes the Lagrange bracket (see \eqref{3.9b}).  Using the boundary condition bases in \eqref{4.68} and the identities in \eqref{4.69}, statements $(i)$--$(v)$ now follow from \cite[Theorems 4.4, 4.5, 4.6, and 4.7]{AKMNR19} after a standard reparametrizion of the self-adjoint extensions \eqref{4.49} and \eqref{4.50} to fit the parametrization used in \cite[Theorem 2.20]{AKMNR19}.
\end{proof}

\begin{remark}
As an illustration of Theorem \ref{t4.3}, we consider the Krein extension, $T_{0,R_K}$, under the additional assumption that $T_{min} \geq \varepsilon I_{_{L^2_r((a,b))}}$ for some 
$\varepsilon > 0$. Then applying \cite[Thm. 3.5 $(ii)$]{FGKLNS21} and Theorem \ref{t4.3} $(iv)$, one computes for the matrix $K_{0,R_K}$ in \eqref{4.61}, 
\begin{equation}\lb{4.25a}
K_{0,R_K}(z) = \begin{pmatrix}
\wti u_1^{\, \prime}(z,b)-\wti u_1^{\, \prime}(0,b) & \wti u_1^{\, \prime}(0,a)-\wti u_1^{\, \prime}(z,a)\\[4mm]
\wti u_2^{\, \prime}(z,b)-\wti u_2^{\, \prime}(0,b) & \wti u_2^{\, \prime}(0,a)-\wti u_2^{\, \prime}(z,a)
\end{pmatrix}, \quad z\in \rho(T_{0,0})\cap\rho(T_{0,R_K}), 
\end{equation}
where we note that $0 \in \sigma(T_{0,R_K})$. \hfill$\diamond$
\end{remark}

\section{Donoghue $m$-functions: One Limit Circle Endpoint} \lb{s5}

In this section we construct the Donoghue $m$-functions in the case where $\tau$ is in the limit circle case at precisely one endpoint (which we choose to be $a$ without loss of generality). We first focus on the Friedrichs extension of 
$T_{min}$ and then use the Krein resolvent formulas from Section \ref{s3} to treat all remaining self-adjoint extensions of $T_{min}$.

Throughout this section we shall assume that Hypothesis \ref{h3.1} holds so that $\tau$ is in the limit circle case at $a$ and in the limit point case at $b$.  We begin by obtaining a general expression for the Donoghue $m$-function of an arbitrary self-adjoint extension $T_\a$ of $T_{min}$ in terms of a unit vector $\phi(i,\dott)\in \cN_i$.  This general expression will then be made more explicit in terms of $\psi(i,\dott)$ (cf.~Hypothesis \ref{h3.1}) in the analysis below.  The Donoghue $m$-function for $T_\a$ is given by (see, e.g., \cite[Eq.~(5.5)]{GNWZ19})
\begin{align}
M_{T_\a,\, \cN_i}^{Do}(z)&=P_{\cN_i}\big(zT_\a+I_{L^2_r((a,b))}\big)(T_\a-zI_{L^2_r((a,b))})^{-1}P_{\cN_i}\big|_{\cN_i}\no\\
&=zI_{\cN_i}+\big(z^2+1\big)P_{\cN_i}(T_\a-zI_{L^2_r((a,b))})^{-1}P_{\cN_i}\big|_{\cN_i},\quad z\in \bbC\backslash\bbR,\lb{6.1}
\end{align}
where $P_{\cN_i}$ denotes the orthogonal projection onto $\cN_i$.  According to \eqref{6.1},
\begin{equation}\lb{6.2}
M_{T_\a,\, \cN_i}^{Do}(z)\in \cB(\cN_i),\ z\in\bbC\backslash\bbR, \text{ and } M_{T_\a,\, \cN_i}^{Do}(\pm i) = \pm i I_{\cN_i}.
\end{equation}
The unit vector $\phi(i,\dott)$ spans the one-dimensional subspace $\cN_i$, so the orthogonal projection onto $\cN_i$ is
\begin{equation}
P_{\cN_i} = (\phi(i,\dott),\dott)_{L^2_r((a,b))}\phi(i,\dott).
\end{equation}
Thus, the action of $M_{T_\a,\, \cN_i}^{Do}(\dott)$ may be computed directly in terms of $\phi(i,\dott)$ as follows:
\begin{align}
&M_{T_\a,\, \cN_i}^{Do}(z)f\lb{6.4}\\
&\quad = \big[zI_{\cN_i} + \big(z^2+1\big)P_{\cN_i}(T_\a-zI_{L^2_r((a,b))})^{-1}P_{\cN_i}\big|_{\cN_i}\big]f\no\\
&\quad = zf + \big(z^2+1\big)P_{\cN_i}(T_\a-zI_{L^2_r((a,b))})^{-1}f\no\\
&\quad = \big(\phi(i,\dott),\big[zI_{\cN_i} + \big(z^2+1\big) (T_\a-zI_{L^2_r((a,b))})^{-1}\big]f\big)_{L_r^2((a,b))}\phi(i,\dott)\no\\
&\quad = \big(\phi(i,\dott),\big[zI_{\cN_i} + \big(z^2+1\big) (T_\a-zI_{L^2_r((a,b))})^{-1}\big]\phi(i,\dott)\big)_{L_r^2((a,b))}\no\\
&\qquad \times(\phi(i,\dott),f)_{L_r^2((a,b))}\phi(i,\dott)\no\\
&\quad = \big[z + \big(z^2+1\big)\big(\phi(i,\dott), (T_\a-zI_{L^2_r((a,b))})^{-1}\phi(i,\dott)\big)_{L_r^2((a,b))}\big]\no\\
&\qquad \times(\phi(i,\dott),f)_{L_r^2((a,b))}\phi(i,\dott),\quad f\in \cN_i,\, z\in \bbC\backslash\bbR,\no
\end{align}
where one uses $f=(\phi(i,\dott),f)_{L^2_r((a,b))}\phi(i,\dott)$ to obtain the fourth equality in \eqref{6.4}.  Hence,
\begin{align}
M_{T_\a,\, \cN_i}^{Do}(z)&= \big[z + \big(z^2+1\big)\big(\phi(i,\dott), (T_\a-zI_{L^2_r((a,b))})^{-1}\phi(i,\dott)\big)_{L_r^2((a,b))}\big]\lb{6.5}\\
&\quad \times(\phi(i,\dott),\dott)_{L_r^2((a,b))}\phi(i,\dott)\big|_{\cN_i},\quad z\in \bbC\backslash\bbR.\no
\end{align}
In order to determine $M_{T_\a,\, \cN_i}^{Do}(\dott)$ in terms of $\psi(i,\dott)$, one must compute the fixed inner product in \eqref{6.5}.  That is, one must compute
\begin{equation}\lb{6.6}
\big(\phi(i,\dott), (T_\a-zI_{L^2_r((a,b))})^{-1}\phi(i,\dott)\big)_{L_r^2((a,b))},\quad z\in \bbC\backslash\bbR.
\end{equation}
In light of \eqref{6.2}, it suffices to compute \eqref{6.6} under the additional assumption that $z\neq \pm i$.  We will first do this for the Dirichlet-type extension $T_0$ (cf.~\eqref{3.2}).

\subsection{The Donoghue \textit{m}-function \textit{M}$\mathstrut_{T_0,\, \cN_i}^{Do}(\dott)$ for \textit{T}$\mathstrut_0$}

Here we shall consider the Dirichlet-type self-adjoint extension $T_0$ of $T_{min}$.  Assuming Hypothesis \ref{h3.1} and taking
\begin{equation}
T_\a=T_0\quad \text{and}\quad \phi(i,\dott):=\|\psi(i,\dott)\|_{L_r^2((a,b))}^{-1}\psi(i,\dott),
\end{equation}
we shall compute the inner product \eqref{6.6} and use \eqref{6.5} to obtain an explicit expression for the Donoghue $m$-function $M_{T_0,\, \cN_i}^{Do}(\dott)$ for $T_0$ in terms of $\psi(i,\dott)$.

For the purposes of evaluating the inner product \eqref{6.6}, we introduce the generalized Cayley transform of $T_{0}$,
\begin{align}
U_{0,z,z'} &= (T_0-z'I_{L^2_r((a,b))})(T_0-zI_{L^2_r((a,b))})^{-1}\lb{6.8}\\
&= I_{L^2_r((a,b))} + (z-z')(T_0-zI_{L^2_r((a,b))})^{-1},\quad z,z'\in \rho(T_0),\no
\end{align}
which forms a bijection from $\cN_{z'}$ to $\cN_z$.  One verifies that
\begin{equation}\lb{6.9}
U_{0,z,z'}\psi(z',\dott) = \psi(z,\dott),\quad z,z'\in \rho(T_0).
\end{equation}
In fact, for fixed $z,z'\in \rho(T_0)$, one uses the fact that $U_{0,z,z'}$ maps into $\cN_z$ to write
\begin{equation}\lb{6.10}
U_{0,z,z'}\psi(z',\dott) = c_0 \psi(z,\dott)
\end{equation}
for some scalar $c_0\in \bbC$.  The second equality in \eqref{6.8} then implies
\begin{equation}
U_{0,z,z'}\psi(z',\dott) = \psi(z',\dott) + (z-z')(T_0-zI_{L^2_r((a,b))})^{-1}\psi(z',\dott),
\end{equation}
so that
\begin{equation}\lb{6.12}
[U_{0,z,z'}\psi(z',\dott)]\;\wti{}\, (a) = \wti \psi(z',a) = 1.
\end{equation}
Taking the generalized boundary value at $a$ throughout \eqref{6.10} and using \eqref{6.12} yields $c_0=1$ in \eqref{6.10}, and \eqref{6.9} follows.

Let $z\in \bbC\backslash\bbR$ with $z\neq \pm i$ be fixed.  Applying \eqref{6.8}, one computes:
\begin{align}
&\big(\phi(i,\dott), (T_0-zI_{L^2_r((a,b))})^{-1}\phi(i,\dott)\big)_{L_r^2((a,b))}\lb{6.13}\\
&\quad = \frac{\big(\psi(i,\dott), (T_0-zI_{L^2_r((a,b))})^{-1}\psi(i,\dott)\big)_{L_r^2((a,b))}}{\|\psi(i,\dott)\|_{L_r^2((a,b))}^{2}}\no\\
&\quad = \frac{(\psi(i,\dott), [U_{0,z,i}-I_{L_r^2((a,b))}]\psi(i,\dott))_{L_r^2((a,b))}}{(z-i)\|\psi(i,\dott)\|_{L_r^2((a,b))}^{2}}\no\\
&\quad = \frac{1}{i-z} + \frac{(\psi(i,\dott), \psi(z,\dott))_{L_r^2((a,b))}}{(z-i)\|\psi(i,\dott)\|_{L_r^2((a,b))}^{2}}.\no
\end{align}
Furthermore, by \eqref{2.1b}  and Theorem \ref{t2.7} $(i)$,
\begin{align}
& (\psi(i,\dott),\psi(z,\dott))_{L_r^2((a,b))} = \int_a^b r(x)dx\, \psi(-i,x)\psi(z,x)\lb{6.22}\\
& \quad =-\frac{W(\psi(-i,\dott),\psi(z,\dott))|_a^b}{z+i} 
= \frac{\wti \psi^{\, \prime}(z,a)-\wti \psi^{\, \prime}(-i,a)}{z+i},\no
\end{align}
where we have used that since $\tau$ is in the limit point case at $b$ and $\psi(-i,\dott),\psi(z,\dott)\in \dom(T_{max})$, an application of Theorem \ref{t2.7} $(i)$ yields
\begin{equation}\lb{6.15}
W(\psi(-i,\dott),\psi(z,\dott))(b) = 0,
\end{equation}
and by Hypothesis \ref{h3.1}, $\wti \psi(-i,a)=\wti \psi(z,a)=1$, so that
\begin{equation}\lb{6.16}
W(\psi(-i,\dott),\psi(z,\dott))(a) = \wti \psi^{\, \prime}(z,a)-\wti \psi^{\, \prime}(-i,a).
\end{equation}
Therefore, \eqref{6.13}--\eqref{6.16} yield
\begin{equation}\lb{6.17}
\big(\phi(i,\dott), (T_0-zI_{L^2_r((a,b))})^{-1}\phi(i,\dott))_{L_r^2((a,b))} 
= \frac{1}{i-z} + \frac{\wti \psi^{\, \prime}(z,a)-\wti \psi^{\, \prime}(-i,a)}{\big(z^2+1\big)\|\psi(i,\dott)\|_{L^2_r((a,b))}^2}.
\end{equation}
By \eqref{2.1b}, Hypothesis \ref{h3.1}, and the limit point assumption at $b$,
\begin{align}
& \|\psi(i,\dott)\|_{L_r^2((a,b))}^2 = \int_a^br(x)dx\, \psi(-i,x)\psi(i,x) = -\frac{W(\psi(-i,\dott),\psi(i,\dott))|_a^b}{2i}\no\\
& \quad = \frac{1}{2i}\big[\wti \psi^{\, \prime}(i,a) - \wti \psi^{\, \prime}(-i,a)\big] 
= \Im\big(\wti \psi^{\, \prime}(i,a) \big).   \lb{6.18}
\end{align}

Applying \eqref{6.17} in \eqref{6.4} and taking simplifications and \eqref{6.18} into account, one obtains the following fact. 

\begin{theorem}\lb{t6.1}
Assume Hypothesis \ref{h3.1}.  The Donoghue $m$-function $M_{T_0,\, \cN_i}^{Do}(\dott):\bbC\backslash\bbR\to \cB(\cN_i)$
for $T_0$ satisfies 
\begin{align}
\begin{split} 
M_{T_0,\, \cN_i}^{Do}(\pm i) &= \pm i I_{\cN_i},     \\
M_{T_0,\, \cN_i}^{Do}(z) &= \Bigg[-i +  \frac{\wti \psi^{\, \prime}(z,a)-\wti \psi^{\, \prime}(-i,a)}{\Im\big(\wti \psi^{\, \prime}(i,a)\big)}\Bigg]I_{\cN_i},\quad z\in \bbC\backslash\bbR,\, z\neq \pm i.\lb{6.20}
\end{split} 
\end{align}
\end{theorem}

\subsection{The Donoghue \textit{m}-function for Self-Adjoint Extensions Other Than \textit{T}$\mathstrut_0$}
The Donoghue $m$-function for $T_0$ was computed explicitly in Theorem \ref{t6.1}.  If $T_{\alpha}$, $\alpha\in (0,\pi)$, is any other self-adjoint extension of $T_{min}$, then the resolvent identity in Theorem \ref{t3.3} may be used to obtain an explicit representation of the Donoghue $m$-function $M_{T_\a,\, \cN_i}^{Do}(\dott)$ for $T_\a$.

\begin{theorem}\lb{t6.2}
Assume Hypothesis \ref{h3.1} and let $\alpha\in (0,\pi)$.  The Donoghue $m$-function $M_{T_\a,\, \cN_i}^{Do}(\dott):\bbC\backslash\bbR\to \cB(\cN_i)$
for $T_{\alpha}$ satisfies 
\begin{align}
M_{T_\a,\, \cN_i}^{Do}(\pm i) &= \pm i I_{\cN_i},   \no \\
M_{T_\a,\, \cN_i}^{Do}(z) &= M_{T_0,\, \cN_i}^{Do}(z)   \lb{6.25} \\
&\quad + (i-z)\frac{\wti \psi^{\, \prime}(z,a)-\wti \psi^{\, \prime}(-i,a)}{\cot(\alpha) + \wti \psi^{\, \prime}(z,a)} (\psi(\overline{z},\dott),\dott)_{L_r^2((a,b))}\psi(i,\dott)\bigg|_{\cN_i},\no\\
&\hspace*{6.8cm}z\in\bbC\backslash\bbR,\, z\neq \pm i.  \no
\end{align}
\end{theorem}
\begin{proof}
Let $\alpha\in (0,\pi)$ be fixed.  By \eqref{6.2}, $M_{T_\a,\, \cN_i}^{Do}(\pm i)=\pm iI_{\cN_i}$.  In order to establish \eqref{6.25}, let $z\in \bbC\backslash\bbR$, $z\neq \pm i$, be fixed.  Considering \eqref{6.1} and invoking \eqref{3.6}, one obtains
\begin{align}
M_{T_\a,\, \cN_i}^{Do}(z) &= M_{T_0,\, \cN_i}^{Do}(z) + \big(z^2+1\big)k_{\alpha}(z)^{-1} (\psi(\overline{z},\dott),\dott)_{L_r^2((a,b))}P_{\cN_i}\psi(z,\dott)\big|_{\cN_i}\lb{6.21}\\
&= M_{T_0,\, \cN_i}^{Do}(z)\no\\
& + \big(z^2+1\big)k_{\alpha}(z)^{-1} (\psi(i,\dott),\psi(z,\dott))_{L_r^2((a,b))} (\psi(\overline{z},\dott),\dott)_{L_r^2((a,b))}\psi(i,\dott)\big|_{\cN_i}.\no
\end{align}
Using \eqref{6.22} in \eqref{6.21}, one obtains
\begin{align}
M_{T_\a,\, \cN_i}^{Do}(z) &= M_{T_0,\, \cN_i}^{Do}(z)\lb{6.23}\\
&\quad + (z-i)\frac{\wti \psi^{\, \prime}(z,a)-\wti \psi^{\, \prime}(-i,a)}{k_{\alpha}(z)} (\psi(\overline{z},\dott),\dott)_{L_r^2((a,b))}\psi(i,\dott)\bigg|_{\cN_i}.\no
\end{align}
Finally, \eqref{6.25} follows from \eqref{6.23} after using the precise form for $k_{\alpha}(z)$ in \eqref{3.4a}.
\end{proof}

\section{Donoghue $m$-functions: Two Limit Circle Endpoints} \lb{s6}

The construction of Donoghue $m$-functions in the case where $\tau$ is in the limit circle case at $a$ and $b$ is the primary aim of this section. Once more we first focus on the Friedrichs extension of $T_{min}$ and then use the Krein resolvent formulas from Section \ref{s4} to treat all remaining self-adjoint extensions of $T_{min}$.

Throughout this section, we shall assume that Hypothesis \ref{h4.1} holds so that $\tau$ is in the limit circle case at $a$ and $b$. We begin by obtaining a general expression for the Donoghue $m$-function of an arbitrary self-adjoint extension $T_{A,B}$ of $T_{min}$ in terms of an orthonormal basis for $\cN_i$.  Recall that the Donoghue $m$-function for $T_{A,B}$ is given by (see, e.g., \cite[Eq.~(5.5)]{GNWZ19})
\begin{align}\lb{5.2}
M_{T_{A,B},\, \cN_i}^{Do}(z)&=P_{\cN_i} (zT_{A,B}+I_{L^2_r((a,b))})(T_{A,B}-zI_{L^2_r((a,b))})^{-1}P_{\cN_i}\big|_{\cN_i}\\
&=zI_{\cN_i}+\big(z^2+1\big)P_{\cN_i}(T_{A,B}-zI_{L^2_r((a,b))})^{-1}P_{\cN_i}\big|_{\cN_i},\quad z\in \bbC\backslash\bbR,  \no
\end{align}
where $P_{\cN_i}$ denotes the orthogonal projection onto $\cN_i$ with $M_{T_{A,B},\, \cN_i}^{Do}(z)\in \cB(\cN_i)$, $z\in \bbC\backslash\bbR$, and
\begin{equation}\lb{5.3}
M_{T_{A,B},\, \cN_i}^{Do}(\pm i) = \pm i I_{\cN_i}.
\end{equation}

Let $\{v_j\}_{j=1,2}$ be an orthonormal basis for the subspace $\cN_i$.  The orthogonal projection onto $\cN_i$ is
\begin{equation}\lb{5.4}
P_{\cN_i} = \sum_{k=1}^2( v_k,\dott)_{L^2_r((a,b))}v_k.
\end{equation}
Therefore, the action of $M_{T_{A,B},\, \cN_i}^{Do}(\dott)$ may be computed directly in terms of \linebreak $\{v_j\}_{j=1,2}$ as follows:
\begin{align}
&M_{T_{A,B},\, \cN_i}^{Do}(z)f 
= \big[zI_{\cN_i} + \big(z^2+1\big)P_{\cN_i}(T_{A,B}-zI_{L^2_r((a,b))})^{-1}P_{\cN_i}\big|_{\cN_i}\big]f  \lb{5.5}   \\
&\quad = zf + \big(z^2+1\big)P_{\cN_i}(T_{A,B}-zI_{L^2_r((a,b))})^{-1}f\no\\
&\quad = \sum_{j=1}^2\big( v_j,\big[zI_{\cN_i} + \big(z^2+1\big) (T_{A,B}-zI_{L^2_r((a,b))})^{-1}\big]f\big)_{L_r^2((a,b))}v_j\no\\
&\quad = \sum_{j,k=1}^2\big(v_j,\big[zI_{\cN_i} + \big(z^2+1\big) (T_{A,B}-zI_{L^2_r((a,b))})^{-1}\big]v_k\big)_{L_r^2((a,b))}(v_k,f)_{L_r^2((a,b))}v_j\no\\
&\quad = \sum_{j,k=1}^2\big[z \delta_{j,k} + \big(z^2+1\big)\big(v_j, (T_{A,B}-zI_{L^2_r((a,b))})^{-1}v_k\big)_{L_r^2((a,b))}\big](v_k,f)_{L_r^2((a,b))}v_j,\no\\
&\hspace*{9.3cm}f\in \cN_i,\, z\in \bbC\backslash\bbR,\no
\end{align}
where one uses $f=\sum_{j=1}^2(v_j,f)_{L^2_r((a,b))}v_j$ to obtain the fourth equality in \eqref{5.5}.  Hence,
\begin{align}
&M_{T_{A,B},\, \cN_i}^{Do}(z)\lb{5.6}\\
&\ = \sum_{j,k=1}^2\big[z \delta_{j,k} + \big(z^2+1\big)\big(v_j, (T_{A,B}-zI_{L^2_r((a,b))})^{-1}v_k\big)_{L_r^2((a,b))}\big](v_k,\dott)_{L_r^2((a,b))}v_j\big|_{\cN_i},\no\\
&\hspace*{11.2cm}z\in \bbC\backslash\bbR.\no
\end{align}
In order to determine $M_{T_{A,B},\, \cN_i}^{Do}(\dott)$ in terms of the orthonormal basis $\{v_j\}_{j=1,2}$, one must compute the fixed inner products in \eqref{5.6}.  That is, one must compute
\begin{equation}\lb{5.7}
\big(v_j, (T_{A,B}-zI_{L^2_r((a,b))})^{-1}v_k\big)_{L_r^2((a,b))},\quad j,k\in\{1,2\},\, z\in \bbC\backslash\bbR.
\end{equation}
In light of \eqref{5.3}, it suffices to compute \eqref{5.7} under the additional assumption that $z\neq \pm i$.  We will first do this for the Dirichlet-type extension $T_{0,0}$ (cf.~\eqref{4.51}).

\subsection{The Donoghue \textit{m}-function \textit{M}$\mathstrut_{T_{0,0},\, \cN_i}^{Do}(\dott)$ for \textit{T}$\mathstrut_{0,0}$}

Here we shall consider the Dirichlet-type self-adjoint extension $T_{0,0}$ of $T_{min}$.  Assuming Hypothesis \ref{h4.1} and taking the orthonormal basis for $\cN_i$ obtained by applying the Gram--Schmidt process to $\{u_j(i,\dott)\}_{j=1,2}$, we shall compute the inner products \eqref{5.7} and use \eqref{5.6} to obtain an explicit expression for the Donoghue $m$-function $M_{T_{0,0},\, \cN_i}^{Do}(\dott)$ for $T_{0,0}$.

In the analysis below, it will be convenient to also introduce an orthonormal basis for $\cN_{-i}$.  To set the stage for applying Gram--Schmidt to $\{u_j(\pm i,\dott)\}_{j=1,2}$, one applies \eqref{2.1b}  and \eqref{4.52}, to compute
\begin{align}
&(u_j(\pm i,\dott),u_k(\pm i,\dott))_{L^2_r((a,b))}  
= \int_a^b r(x)dx\, \overline{u_j(\pm i,x)}u_k(\pm i, x)\no\\
&\quad = \int_a^b r(x)dx\, u_j(\mp i,x)u_k(\pm i,x)   
= \frac{W\big(u_j(\mp i,\dott),u_k(\pm i,\dott)\big)\big|_a^b}{\mp i - (\pm i)}\no\\
&\quad = \mp \frac{1}{2i}W\big(u_j(\mp i,\dott),u_k(\pm i,\dott)\big)\big|_a^b\no\\
&\quad = \mp \frac{1}{2i}\big\{ \wti u_j(\mp i,b)\wti u_k^{\, \prime}(\pm i,b) - \wti u_j^{\, \prime}(\mp i,b)\wti u_k(\pm i,b)\no\\
&\hspace*{1.6cm} -\big[\wti u_j(\mp i,a)\wti u_k^{\, \prime}(\pm i,a) - \wti u_j^{\, \prime}(\mp i,a)\wti u_k(\mp i,a)\big]\big\}\no\\
&\quad = \mp \frac{1}{2i}\big\{ \wti u_k^{\, \prime}(\pm i,b)\delta_{j,1} - \wti u_j^{\, \prime}(\mp i,b)\delta_{k,1}\no\\
&\hspace*{1.6cm} -\big[\wti u_k^{\, \prime}(\pm i,a)\delta_{j,2} - \wti u_j^{\, \prime}(\mp i,a)\delta_{k,2}\big]\big\},\quad j,k\in\{1,2\}.     \lb{5.8}
\end{align}
In particular, \eqref{5.8} implies
\begin{align}
& (u_1(\pm i,\dott),u_2(\pm i,\dott))_{L^2_r((a,b))} = \mp \frac{1}{2i}\big[\wti u_2^{\, \prime}(\pm i,b) + \wti u_1^{\, \prime}(\mp i,a)\big]   \no \\
& \quad =\mp \frac{1}{2i}\big[\wti u_2^{\, \prime}(\pm i,b) - \wti u_2^{\, \prime}(\mp i,b)\big] 
=\mp \frac{1}{2i}\big[\wti u_2^{\, \prime}(\pm i,b) - \overline{\wti u_2^{\, \prime}(\pm i,b)}\big]\no\\
&\quad =\mp \Im\big(\wti u_2^{\, \prime}(\pm i,b)\big) 
= (u_2(\pm i,\dott),u_1(\pm i,\dott))_{L^2_r((a,b))},     \lb{5.9}
\end{align}
and
\begin{align}
& \big\|u_1(\pm i,\dott)\big\|_{L^2_r((a,b))}^2 = \mp \frac{1}{2i}\big[\wti u_1^{\, \prime}(\pm i,b)-\wti u_1^{\, \prime}(\mp i,b)\big]     \no \\
&\quad = \mp \frac{1}{2i}\big[\wti u_1^{\, \prime}(\pm i,b) - \overline{\wti u_1^{\, \prime}(\pm i,b)}\big] 
= \mp \Im\big(\wti u_1^{\, \prime}(\pm i,b) \big),     \lb{5.10} \\
& \big\|u_2(\pm i,\dott)\big\|_{L^2_r((a,b))}^2 = \pm \frac{1}{2i}\big[\wti u_2^{\, \prime}(\pm i,a)-\wti u_2^{\, \prime}(\mp i,a)\big]     \no \\
& \quad = \pm \frac{1}{2i}\big[\wti u_2^{\, \prime}(\pm i,a) - \overline{\wti u_2^{\, \prime}(\pm i,a)}\big] 
= \pm \Im\big(\wti u_2^{\, \prime}(\pm i,a) \big).     \lb{5.10a}
\end{align}
Applying the Gram--Schmidt process to $\{u_j(\pm i,\dott)\}_{j=1,2}$ then yields an orthonormal basis $\{v_j(\pm i,\dott)\}_{j=1,2}$ for $\cN_{\pm i}$ as follows:
\begin{align}
v_1(\pm i,\dott)& =c_1(\pm i)u_1(\pm i,\dott),\lb{5.11}\\
v_2(\pm i,\dott)& = c_2(\pm i)\bigg[u_2(\pm i,\dott)-\frac{(u_1(\pm i,\dott),u_2(\pm i,\dott))_{L^2_r((a,b))}}{\|u_1(\pm i,\dott)\|_{L^2_r((a,b))}^{2}}u_1(\pm i,\dott)\bigg]\lb{5.12}\\
&\;= c_2(\pm i)\bigg[u_2(\pm i,\dott)-\frac{\Im\big(\wti u_2^{\, \prime}(i,b)\big)}{\Im\big(\wti u_1^{\, \prime}(i,b) \big)}u_1(\pm i,\dott)\bigg],\no
\end{align}
where
\begin{align}
c_1(\pm i)& = \|u_1(\pm i,\dott)\|_{L^2_r((a,b))}^{-1}=\big[\mp \Im\big(\wti u_1^{\, \prime}(\pm i,b) \big)\big]^{-1/2},\lb{5.13}\\
c_2(\pm i)& = \bigg\|u_2(\pm i,\dott)-\frac{\Im\big(\wti u_2^{\, \prime}(i,b)\big)}{\Im\big(\wti u_1^{\, \prime}(i,b) \big)}u_1(\pm i,\dott)\bigg\|_{L^2_r((a,b))}^{-1}\lb{5.14}\\
&\;= \bigg[\pm \Im\big(\wti u_2^{\, \prime}(\pm i,a) \big)\pm\frac{\big[\Im\big(\wti u_2^{\, \prime}(\pm i,b)\big)\big]^2}{\Im\big(\wti u_1^{\, \prime}(\pm i,b) \big)}  \bigg]^{-1/2}, \no
\end{align}
and the equality $\Im\big(\wti u_2^{\, \prime}(-i,b)\big)\big/\Im\big(\wti u_1^{\, \prime}(-i,b) \big) = \Im\big(\wti u_2^{\, \prime}(i,b)\big)\big/\Im\big(\wti u_1^{\, \prime}(i,b) \big)$ has been applied.  Based on \eqref{4.5}, one infers that
\begin{equation}\lb{5.15}
c_j(i)=c_j(-i),\quad j\in \{1,2\}.
\end{equation}
In addition, by taking conjugates throughout \eqref{5.11}--\eqref{5.14} and applying \eqref{4.5}, one obtains
\begin{equation}\lb{5.16}
\overline{v_j(\pm i,\dott)} = v_j(\mp i,\dott),\quad j\in \{1,2\}.
\end{equation}
Taking the orthonormal basis $\{v_j(i,\dott)\}_{j=1,2}$ for $\cN_i$ in \eqref{5.6} then yields the following expression for the Donoghue $m$-function $M_{T_{0,0},\, \cN_i}^{Do}(\dott)$ for $T_{0,0}$:
\begin{align}
&M_{T_{0,0},\, \cN_i}^{Do}(z)\lb{5.17}\\
&\quad = \sum_{j,k=1}^2\bigg[z \delta_{j,k} + \big(z^2+1\big)\big(v_j(i,\dott), (T_{0,0}-zI_{L_r^2((a,b))})^{-1}v_k(i,\dott)\big)_{L_r^2((a,b))}\big]\no\\
&\hspace*{2cm} \times (v_k(i,\dott),\dott)_{L_r^2((a,b))}v_j(i,\dott)\big|_{\cN_i},\quad z\in \bbC\backslash\bbR.\no
\end{align}
In the special cases $z=\pm i$, one obtains (cf.~\eqref{5.3})
\begin{equation}\lb{5.18}
M_{T_{0,0},\, \cN_i}^{Do}(\pm i) = \pm i I_{\cN_i}.
\end{equation}
Thus, to obtain an explicit representation for $M_{T_{0,0},\, \cN_i}^{Do}(\dott)$, it remains to evaluate the inner products
\begin{equation}\lb{5.19}
\begin{split}
\big(v_j(i,\dott), (T_{0,0}-zI_{L_r^2((a,b))})^{-1}v_k(i,\dott)\big)_{L_r^2((a,b))},\quad j,k\in\{1,2\},&\\
z\in \bbC\backslash\bbR,\, z\neq \pm i.&
\end{split}
\end{equation}
For the purposes of evaluating the inner products \eqref{5.19}, we introduce the generalized Cayley transform of $T_{0,0}$,
\begin{align}
U_{0,0,z,z'} &= (T_{0,0}-z'I_{L^2_r((a,b))})(T_{0,0}-zI_{L^2_r((a,b))})^{-1}\lb{5.20}\\
&= I_{L^2_r((a,b))} + (z-z')(T_{0,0}-zI_{L^2_r((a,b))})^{-1},\quad z,z'\in \rho(T_{0,0}),\no
\end{align}
which forms a bijection from $\cN_{z'}$ to $\cN_z$.  One verifies that
\begin{equation}\lb{5.21}
U_{0,0,z,z'}u_j(z',\dott) = u_j(z,\dott),\quad j\in \{1,2\},\, z,z'\in \rho(T_{0,0}).
\end{equation}
In fact, for fixed $z,z'\in \rho(T_{0,0})$, one uses the fact that $U_{0,0,z,z'}$ maps into $\cN_z$ to write
\begin{equation}\lb{5.22}
U_{0,0,z,z'}u_j(z',\dott) = \alpha_{j,1}u_1(z,\dott) + \alpha_{j,2}u_2(z,\dott),\quad j\in \{1,2\},
\end{equation}
for some scalars $\alpha_{j,k}\in \bbC$, $j,k\in \{1,2\}$.  The second equality in \eqref{5.20} then implies
\begin{equation}\lb{5.23}
\begin{split}
U_{0,0,z,z'}u_j(z',\dott) = u_j(z',\dott) + (z-z')(T_{0,0}-zI_{L_r^2((a,b))})^{-1}u_j(z',\dott),&\\
j\in \{1,2\},&
\end{split}
\end{equation}
so that
\begin{equation}\lb{5.24}
[U_{0,0,z,z'}u_j(z',\dott)]\; \wti{}\, (x) = \wti u_j(z',x),\quad x\in \{a,b\},\, j\in \{1,2\}.
\end{equation}
Evaluating \eqref{5.22} and \eqref{5.24} at $a$ yields $\alpha_{1,2}=0$ and $\alpha_{2,2}=1$.  Similarly, evaluating \eqref{5.22} and \eqref{5.24} at $b$ yields $\alpha_{1,1}=1$ and $\alpha_{2,1}=0$.  Hence, \eqref{5.21} follows.

We will now calculate the inner products \eqref{5.19}.  Let
\begin{equation}\lb{5.25}
\text{$z\in \bbC\backslash \bbR$ be fixed with $z\neq \pm i$}.
\end{equation}
The system $\{v_j(z,\dott)\}_{j=1,2}$ defined by 
\begin{equation}\lb{5.26}
v_j(z,\dott) = U_{0,0,z,i}v_j(i,\dott),\quad j\in\{1,2\},
\end{equation}
is a basis for the subspace $\cN_z$.  Applying \eqref{5.11}--\eqref{5.12} and \eqref{5.21} in \eqref{5.26}, one obtains
\begin{equation}\lb{5.27}
\begin{split}
v_1(z,\dott)&=c_1(i)u_1(z,\dott),\\
v_2(z,\dott)&= c_2(i)\Bigg[u_2(z,\dott)-\frac{\Im\big(\wti u_2^{\, \prime}( i,b)\big)}{\Im\big(\wti u_1^{\, \prime}( i,b) \big)}u_1(z,\dott)\Bigg].
\end{split}
\end{equation}
The inner products \eqref{5.19} can be recast in terms of $\{v_j(z,\dott)\}_{j=1,2}$ as follows:
\begin{align}
&\big(v_j(i,\dott),(T_{0,0}-zI_{L^2_r((a,b))})^{-1}v_k(i,\dott) \big)_{L^2_r((a,b))}\no\\
&\quad =\frac{1}{z-i} (v_j(i,\dott),[U_{0,0,z,i}-I_{L^2_r((a,b))}]v_k(i,\dott))_{L^2_r((a,b))}\lb{5.28}\\
&\quad = \frac{1}{i-z}\delta_{j,k} + \frac{1}{z-i}(v_j(i,\dott),v_k(z,\dott))_{L^2_r((a,b))},\quad j,k\in\{1,2\}.\no
\end{align}
In turn, by \eqref{2.1b}  and \eqref{5.16}, one obtains
\begin{align}
(v_j(i,\dott),v_k(z,\dott))_{L^2_r((a,b))}&= \int_a^b r(x)dx\, v_j(-i,x)v_k(z,x)\no\\
&=-\frac{W\big(v_j(-i,\dott),v_k(z,\dott)\big)\big|_a^b}{z+i},\quad j,k\in \{1,2\}.\lb{5.29}
\end{align}
Using \eqref{5.29}, one recasts \eqref{5.28} as
\begin{equation}\lb{5.30}
\begin{split}
&\big(v_j(i,\dott),(T_{0,0}-zI_{L^2_r((a,b))})^{-1}v_k(i,\dott) \big)_{L^2_r((a,b))}\\
&\quad = \frac{1}{i-z}\delta_{j,k} - \frac{W\big(v_j(-i,\dott),v_k(z,\dott)\big)\big|_a^b}{1+z^2},\quad j,k\in\{1,2\}.
\end{split}
\end{equation}
After substituting \eqref{5.30} in \eqref{5.17} and taking cancellations into account, one obtains
\begin{align}
&M_{T_{0,0},\, \cN_i}^{Do}(z)\lb{5.31}\\
&\quad = \sum_{j,k=1}^2 \big[-i \delta_{j,k} - W\big(v_j(-i,\dott),v_k(z,\dott)\big)\big|_a^b\big](v_k(i,\dott),\dott)_{L_r^2((a,b))}
v_j(i,\dott)\big|_{\cN_i},\no\\
&\hspace*{9cm} z\in \bbC\backslash\bbR,\, z\neq \pm i.\no
\end{align}
The Wronskians
\begin{equation}\lb{5.32}
W_{j,k}(z):=W(v_j(-i,\dott),v_k(z,\dott))\big|_a^b,\quad z\in \bbC\backslash\bbR,\, z\neq \pm i,
\end{equation}
that appear in \eqref{5.31} can be computed by applying \eqref{4.52} and \eqref{5.27}.  One obtains for $z\in \bbC\backslash\bbR$, $z\neq \pm i$:
\begin{align}
&W_{1,1}(z)= [c_1(i)]^2\big[\wti u_1^{\, \prime}(z,b)-\wti u_1^{\, \prime}(-i,b)\big],\lb{5.33}\\
&W_{1,2}(z)= c_1(i)c_2(i)\Bigg\{\frac{\Im\big(\wti u_2^{\, \prime}( i,b)\big)}{\Im\big(\wti u_1^{\, \prime}( i,b) \big)}\big[\wti u_1^{\, \prime}(-i,b)-\wti u_1^{\, \prime}(z,b)\big]\lb{5.34}\\
&\hspace*{3.3cm}+\wti u_2^{\, \prime}(z,b)+\wti u_1^{\, \prime}(-i,a)\Bigg\}, \no\\
&W_{2,1}(z)= -c_1(i)c_2(i)\Bigg\{\frac{\Im\big(\wti u_2^{\, \prime}( i,b)\big)}{\Im\big(\wti u_1^{\, \prime}(i,b) \big)}\big[\wti u_1^{\, \prime}(z,b) - \wti u_1^{\, \prime}(-i,b)\big]\lb{5.35}\\
&\hspace*{3.5cm} + \wti u_2^{\, \prime}(-i,b)+\wti u_1^{\, \prime}(z,a)\Bigg\},\no\\
&W_{2,2}(z) = [c_2(i)]^2\Bigg\{\Bigg[\wti u_2^{\, \prime}(-i,b)-\wti u_2^{\, \prime}(z,b)\lb{5.36}\\
&\hspace*{3.1cm}+\frac{\Im\big(\wti u_2^{\, \prime}(i,b)\big)}{\Im\big(\wti u_1^{\, \prime}(i,b) \big)}\big[\wti u_1^{\, \prime}(z,b)-\wti u_1^{\, \prime}(-i,b)\big]\Bigg]\frac{\Im\big(\wti u_2^{\, \prime}(i,b)\big)}{\Im\big(\wti u_1^{\, \prime}(i,b) \big)}\no\\
&\hspace*{3cm} +\wti u_2^{\, \prime}(-i,a) - \wti u_2^{\, \prime}(z,a)+\frac{\Im\big(\wti u_2^{\, \prime}(i,b)\big)}{\Im\big(\wti u_1^{\, \prime}(i,b) \big)}\big[\wti u_1^{\, \prime}(z,a)-\wti u_1^{\, \prime}(-i,a)\big]\Bigg\}.\no
\end{align}
The relations \eqref{5.18} and \eqref{5.31}--\eqref{5.36} now yield an explicit representation for the Donoghue $m$-function $M_{T_{0,0},\, \cN_i}^{Do}(\dott)$ for $T_{0,0}$.

\begin{theorem}\lb{t5.1}
Assume Hypothesis \ref{h4.1} and let $\{v_j(i,\dott)\}_{j=1,2}$ be the orthonormal basis for $\cN_i$ defined in \eqref{5.11}--\eqref{5.14}.  The Donoghue $m$-function $M_{T_{0,0},\, \cN_i}^{Do}(\dott):\bbC\backslash\bbR\to \cB(\cN_i)$ for $T_{0,0}$ satisfies 
\begin{align}
M_{T_{0,0},\, \cN_i}^{Do}(\pm i) &= \pm i I_{\cN_i},      \no \\
M_{T_{0,0},\, \cN_i}^{Do}(z) &= -\sum_{j,k=1}^2 [i \delta_{j,k} + W_{j,k}(z)] (v_k(i,\dott),\dott)_{L_r^2((a,b))}v_j(i,\dott)\big|_{\cN_i},   \lb{5.37}   \\
&= -iI_{\cN_i}-\sum_{j,k=1}^2W_{j,k}(z)\big(v_k(i,\dott),\dott\big)_{L_r^2((a,b))}v_j(i,\dott)\big|_{\cN_i},  \no  \\
& \hspace*{5.7cm} z\in \bbC\backslash\bbR,\, z\neq \pm i,\no
\end{align}
where the matrix $\big(W_{j,k}(\dott)\big)_{j,k=1}^2$ is given by \eqref{5.33}--\eqref{5.36}.
\end{theorem}

\subsection{The Donoghue \textit{m}-function for Self-Adjoint Extensions Other Than \textit{T}$\mathstrut_{0,0}$}

The Donoghue $m$-function $M_{T_{0,0},\, \cN_i}^{Do}(\dott)$ for $T_{0,0}$ was computed explicitly in Theorem \ref{t5.1}.  If $T_{A,B}$ is any other self-adjoint extension of $T_{min}$, then the resolvent identities in Theorem \ref{t4.3} may be used to obtain an explicit representation of the Donoghue $m$-function for $T_{A,B}$.

We begin with the case when either $T_{A,B}=T_{\alpha,\beta}$ for $\alpha,\beta\in (0,\pi)$ or $T_{A,B}=T_{\varphi,R}$ for some $\varphi\in [0,2\pi)$, $R\in \SL(2,\bbR)$, with $R_{1,2}\neq 0$.  In this case, items $(i)$ and $(iv)$ in Theorem \ref{t4.3} imply
\begin{align}
(T_{A,B}-zI_{L_r^2((a,b))})^{-1} &= (T_{0,0}-zI_{L_r^2((a,b))})^{-1}\lb{5.38}\\
&\quad + \sum_{j,k=1}^2 \big[K_{A,B}(z)^{-1}\big]_{j,k} ( u_j(\ol{z},\dott),\,\cdot\,)_{L^2_r((a,b))} u_k(z,\dott),\no\\
&\hspace*{4.9cm} z\in \rho(T_{0,0})\cap\rho(T_{A,B}),\no
\end{align}
where $K_{A,B}(\dott)=K_{\alpha,\beta}(\dott)$ or $K_{A,B}(\dott)=K_{\varphi,R}(\dott)$ (cf.~\eqref{4.53} and \eqref{4.61}) according to whether $T_{A,B}=T_{\alpha,\beta}$ or $T_{A,B}=T_{\varphi,R}$, respectively.  Employing \eqref{5.38} in \eqref{5.2}, one obtains the following representation for the Donoghue $m$-function $M_{T_{A,B},\, \cN_i}^{Do}(\dott)$ of $T_{A,B}$:
\begin{align}
&M_{T_{A,B},\, \cN_i}^{Do}(z)\lb{5.39}\\
&\  = zI_{\cN_i} + \big(z^2+1\big)P_{\cN_i}(T_{0,0}-zI_{L_r^2((a,b))})^{-1}P_{\cN_i}\big|_{\cN_i}\no\\
&\qquad +\big(z^2+1\big)\bigg[\sum_{j,k=1}^2 \big[K_{A,B}(z)^{-1}\big]_{j,k} ( u_j(\ol{z},\dott),\,\cdot\,)_{L^2_r((a,b))} P_{\cN_i} u_k(z,\dott)\bigg]\bigg|_{\cN_i}\no\\
&\ = M_{T_{0,0},\, \cN_i}^{Do}(z) +\big(z^2+1\big)\sum_{j,k=1}^2 \big[K_{A,B}(z)^{-1}\big]_{j,k} ( u_j(\ol{z},\dott),\,\cdot\,)_{L^2_r((a,b))} P_{\cN_i} u_k(z,\dott)\big|_{\cN_i},\no\\
&\hspace*{10.9cm} z\in \bbC\backslash\bbR.\no
\end{align}
In light of \eqref{5.3}, to obtain a final expression for $M_{T_{A,B},\, \cN_i}^{Do}(\dott)$, one must compute $P_{\cN_i}u_k(z,\dott)$, $k\in \{1,2\}$, for $z\in \bbC\backslash\bbR$, $z\neq \pm i$.  Let $z\in \bbC\backslash\bbR$, $z\neq \pm i$.  Invoking the orthonormal basis $\{v_j(i,\dott)\}_{j=1,2}$ for $\cN_i$ defined in \eqref{5.11}--\eqref{5.14}, one obtains
\begin{equation}\lb{5.40}
P_{\cN_i}u_k(z,\dott) = \sum_{\ell=1}^2 (v_\ell(i,\dott),u_k(z,\dott))_{L_r^2((a,b))}v_{\ell}(i,\dott),\quad k\in \{1,2\}.
\end{equation}
By \eqref{2.1b} ,
\begin{align}
\big(v_\ell(i,\dott),u_k(z,\dott)\big)_{L_r^2((a,b))}&=\int_a^br(x)dx\, v_{\ell}(-i,x)u_k(z,x)\lb{5.41}\\
&= -\frac{W(v_{\ell}(-i,\dott),u_k(z,\dott))|_a^b}{z+i},\quad \ell,k\in \{1,2\}.\no
\end{align}
The Wronskians
\begin{equation}\lb{5.42}
W_{\ell,k}^{Kr}(z):=W(v_{\ell}(-i,\dott),u_k(z,\dott))|_a^b,\quad \ell,k\in \{1,2\},
\end{equation}
that appear in \eqref{5.41} can be computed by applying \eqref{4.52} and \eqref{5.11}--\eqref{5.12}.  One obtains:
\begin{align}
W_{1,1}^{Kr}(z)&=c_1(i)\big[\wti u_1^{\, \prime}(z,b) - \wti u_1^{\, \prime}(-i,b) \big],\lb{5.43}\\
W_{1,2}^{Kr}(z)&=c_1(i)\big[\wti u_2^{\, \prime}(z,b) + \wti u_1^{\, \prime}(-i,a) \big],\lb{5.44}\\
W_{2,1}^{Kr}(z)&=\wti v_2(-i,b)\wti u_1^{\, \prime}(z,b) - \wti v_2^{\, \prime}(-i,b) - \wti v_2(-i,a)\wti u_1^{\,\prime}(z,a)\lb{5.45}\\
&=-c_2(i)\bigg\{\frac{\Im\big(\wti u_2^{\, \prime}(i,b)\big)}{\Im\big(\wti u_1^{\, \prime}(i,b) \big)}\big[\wti u_1^{\, \prime}(z,b) - \wti u_1^{\, \prime}(-i,b)\big]+\wti u_2^{\, \prime}(-i,b)+\wti u_1^{\,\prime}(z,a)\bigg\},\no\\
W_{2,2}^{Kr}(z)&=\wti v_2(-i,b)\wti u_2^{\, \prime}(z,b) - \wti v_2(-i,a)\wti u_2^{\, \prime}(z,a) + \wti v_2^{\, \prime}(-i,a)\lb{5.46}\\
&=-c_2(i)\bigg\{\frac{\Im\big(\wti u_2^{\, \prime}(i,b)\big)}{\Im\big(\wti u_1^{\, \prime}(i,b) \big)}\big[\wti u_2^{\, \prime}(z,b) + \wti u_1^{\, \prime}(-i,a)\big] + \wti u_2^{\, \prime}(z,a) - \wti u_2^{\, \prime}(-i,a)\bigg\}.\no
\end{align}
Therefore, \eqref{5.40} may be recast as
\begin{equation}\lb{5.47}
P_{\cN_i}u_k(z,\dott) = -\frac{1}{z+i}\sum_{\ell=1}^2 W_{\ell,k}^{Kr}(z)v_{\ell}(i,\dott),\quad k\in \{1,2\}.
\end{equation}
By combining \eqref{5.39} and \eqref{5.47}, one obtains
\begin{align}
M_{T_{A,B},\, \cN_i}^{Do}(z)&= M_{T_{0,0},\, \cN_i}^{Do}(z)\\
&\ + (i-z)\sum_{j,k,\ell=1}^2 \big[K_{A,B}(z)^{-1}\big]_{j,k}W_{\ell,k}^{Kr}(z) ( u_j(\ol{z},\dott),\,\cdot\,)_{L^2_r((a,b))} v_{\ell}(i,\dott)\big|_{\cN_i}.\no
\end{align}
These considerations are summarized next.

\begin{theorem}\lb{t5.2}
Assume Hypothesis \ref{h4.1} and let $\{v_j(i,\dott)\}_{j=1,2}$ be the orthonormal basis for $\cN_i$ defined in \eqref{5.11}--\eqref{5.14}.  The following items $(i)$ and $(ii)$ hold.\\[2mm]
$(i)$ If $\alpha,\beta\in (0,\pi)$, then the Donoghue $m$-function $M_{T_{\a,\b},\, \cN_i}^{Do}(\dott):\bbC\backslash\bbR\to \cB(\cN_i)$ for $T_{\alpha,\beta}$ satisfies 
\begin{align}
M_{T_{\a,\b},\, \cN_i}^{Do}(\pm i) &= \pm i I_{\cN_i},    \no \\ 
M_{T_{\a,\b},\, \cN_i}^{Do}(z)&= M_{T_{0,0},\, \cN_i}^{Do}(z)\\
&+ (i-z)\sum_{j,k,\ell=1}^2 \big[K_{\alpha,\beta}(z)^{-1}\big]_{j,k}W_{\ell,k}^{Kr}(z) ( u_j(\ol{z},\dott),\,\cdot\,)_{L^2_r((a,b))} v_{\ell}(i,\dott)\big|_{\cN_i},\no\\
&\hspace*{7.9cm} z\in\bbC\backslash\bbR,\, z\neq\pm i,\no
\end{align}
where the matrices $K_{\alpha,\beta}(\dott)$ and $\big(W_{\ell,k}^{Kr}(\dott)\big)_{\ell,k=1}^2$ are given by \eqref{4.53} and \eqref{5.43}--\eqref{5.46}, respectively.\\[2mm]
$(ii)$ If $\varphi\in[0,2\pi)$ and $R\in \SL(2,\bbR)$ with $R_{1,2}\neq 0$, then the Donoghue $m$-function $M_{T_{\varphi,R},\, \cN_i}^{Do}(\dott):\bbC\backslash\bbR\to \cB(\cN_i)$ for $T_{\varphi,R}$ satisfies
\begin{align}
M_{T_{\varphi,R},\, \cN_i}^{Do}(\pm i) &= \pm i I_{\cN_i},    \no \\ 
M_{T_{\varphi, R},\, \cN_i}^{Do}(z)&= M_{T_{0,0},\, \cN_i}^{Do}(z)\\
&\hspace{-.1cm} + (i-z)\sum_{j,k,\ell=1}^2 \big[K_{\varphi,R}(z)^{-1}\big]_{j,k}W_{\ell,k}^{Kr}(z) ( u_j(\ol{z},\dott),\,\cdot\,)_{L^2_r((a,b))} v_{\ell}(i,\dott)\big|_{\cN_i},\no\\
&\hspace*{7.9cm} z\in\bbC\backslash\bbR,\, z\neq\pm i,\no
\end{align}
where the matrices $K_{\varphi,R}(\dott)$ and $\big(W_{\ell,k}^{Kr}(\dott)\big)_{\ell,k=1}^2$ are given by \eqref{4.61} and \eqref{5.43}--\eqref{5.46}, respectively.
\end{theorem}

It remains to compute the Donoghue $m$-functions for $T_{0,\beta}$ and $T_{\alpha,0}$ with $\alpha,\beta\in (0,\pi)$ and $T_{\varphi,R}$ for $\varphi\in[0,2\pi)$ and $R\in\SL(2,\bbR)$ with $R_{1,2}=0$.

\begin{theorem}
Assume Hypothesis \ref{h4.1} and let $\{v_j(i,\dott)\}_{j=1,2}$ be the orthonormal basis for $\cN_i$ defined in \eqref{5.11}--\eqref{5.14}.  The following items $(i)$ and $(ii)$ hold.\\[2mm]
$(i)$ If $\alpha\in (0,\pi)$, then the Donoghue $m$-function $M_{T_{\a,0},\, \cN_i}^{Do}(\dott):\bbC\backslash\bbR\to \cB(\cN_i)$ for $T_{\alpha,0}$ satisfies 
\begin{align}
M_{T_{\a,0},\, \cN_i}^{Do}(\pm i) &= \pm i I_{\cN_i},    \no \\
M_{T_{\a,0},\, \cN_i}^{Do}(z) &= M_{T_{0,0},\, \cN_i}^{Do}(z)\lb{5.53}\\
&\quad + \frac{z-i}{\cot(\alpha)+\wti u_2^{\, \prime}(z,a)} ( u_2(\ol{z},\dott),\,\cdot\,)_{L^2_r((a,b))} \sum_{\ell=1}^2 W_{\ell,2}^{Kr}(z)v_{\ell}(i,\dott)\big|_{\cN_i},\no\\
&\hspace*{7.28cm} z\in\bbC\backslash\bbR,\, z\neq\pm i,\no
\end{align}
where the scalars $\big\{W_{\ell,2}^{Kr}(\dott)\big\}_{\ell=1,2}$ are given by \eqref{5.44} and \eqref{5.46}.\\[2mm]
$(ii)$ If $\beta\in (0,\pi)$, then the Donoghue $m$-function $M_{T_{0,\b},\, \cN_i}^{Do}(\dott):\bbC\backslash\bbR\to \cB(\cN_i)$ for $T_{0,\beta}$ satisfies 
\begin{align}
M_{T_{0,\b},\, \cN_i}^{Do}(\pm i) &= \pm i I_{\cN_i},     \no \\
M_{T_{0,\b},\, \cN_i}^{Do}(z) &= M_{T_{0,0},\, \cN_i}^{Do}(z)\lb{5.56}\\
&\quad - \frac{z-i}{\cot(\beta)+\wti u_1^{\, \prime}(z,b)} ( u_1(\ol{z},\dott),\,\cdot\,)_{L^2_r((a,b))} \sum_{\ell=1}^2 W_{\ell,1}^{Kr}(z)v_{\ell}(i,\dott)\big|_{\cN_i},\no\\
&\hspace*{7.28cm} z\in\bbC\backslash\bbR,\, z\neq\pm i,\no
\end{align}
where the scalars $\big\{W_{\ell,1}^{Kr}(\dott)\big\}_{\ell=1,2}$ are given by \eqref{5.43} and \eqref{5.45}.\\[2mm]
$(iii)$  If $\varphi \in [0,2\pi)$ and $R\in \SL(2,\bbR)$ with $R_{1,2}=0$, then the Donoghue $m$-function $M_{T_{\varphi,R},\, \cN_i}^{Do}(\dott):\bbC\backslash\bbR\to \cB(\cN_i)$ for $T_{\varphi,R}$ satisfies 
\begin{align}
& M_{T_{\varphi,R},\, \cN_i}^{Do}(\pm i)=\pm i I_{\cN_i},     \no \\
&M_{T_{\varphi,R},\, \cN_i}^{Do}(z) = M_{T_{0,0},\, \cN_i}^{Do}(z)     \lb{5.59} \\
&\quad- \frac{z-i}{k_{\varphi,R}(z)} (u_{\varphi,R}(\ol{z},\dott),\dott)_{L^2_r((a,b))}\sum_{\ell=1}^2\big[e^{-i\varphi}R_{2,2}W_{\ell,2}^{Kr}(z)+W_{\ell,1}^{Kr}(z) \big]v_{\ell}(i,\dott)\big|_{\cN_i},\no\\
&\hspace*{9.5cm} z\in\bbC\backslash\bbR,\, z\neq\pm i,\no
\end{align}
where the scalar $k_{\varphi,R}(\dott)$ and the matrix $\big(W_{\ell,k}^{Kr}(\dott)\big)_{\ell,k=1}^2$ are given by \eqref{4.64} and \eqref{5.43}--\eqref{5.46}, respectively.
\end{theorem}
\begin{proof}
To prove item $(i)$, let $\alpha\in (0,\pi)$.  By \eqref{5.3}, $M_{T_{\a,0},\, \cN_i}^{Do}(\pm i) = \pm iI_{\cN_i}$.  In order to establish \eqref{5.53}, let $z\in \bbC\backslash\bbR$, $z\neq \pm i$, be fixed.  Taking $T_{A,B}=T_{\alpha,0}$ in \eqref{5.2} and invoking \eqref{4.60}, one obtains
\begin{equation}\lb{5.52}
M_{T_{\a,0},\, \cN_i}^{Do}(z) = M_{T_{0,0},\, \cN_i}^{Do}(z) + \big(z^2+1\big)K_{\alpha,0}(z)^{-1} ( u_2(\ol{z},\dott),\,\cdot\,)_{L^2_r((a,b))} P_{\cN_i}u_2(z,\dott)\big|_{\cN_i}.
\end{equation}
Using \eqref{5.47} with $k=2$ in \eqref{5.52}, one obtains
\begin{align}
M_{T_{\a,0},\, \cN_i}^{Do}(z) &= M_{T_{0,0},\, \cN_i}^{Do}(z)\lb{5.54}\\
&\quad + (i-z)K_{\alpha,0}(z)^{-1} ( u_2(\ol{z},\dott),\,\cdot\,)_{L^2_r((a,b))} \sum_{\ell=1}^2 W_{\ell,2}^{Kr}(z)v_{\ell}(i,\dott)\big|_{\cN_i}.\no
\end{align}
Hence, \eqref{5.53} follows from \eqref{5.54} by applying the precise form for $K_{\alpha,0}(z)$ given in \eqref{4.59}.  This completes the proof of item $(i)$.

To prove item $(ii)$, let $\beta\in (0,\pi)$.  By \eqref{5.3}, $M_{T_{0,\b},\, \cN_i}^{Do}(\pm i) = \pm iI_{\cN_i}$.  In order to establish \eqref{5.56}, let $z\in \bbC\backslash\bbR$, $z\neq \pm i$, be fixed.  Taking $T_{A,B}=T_{0,\beta}$ in \eqref{5.2} and invoking \eqref{4.57}, one obtains
\begin{equation}\lb{5.55}
M_{T_{0,\b},\, \cN_i}^{Do}(z) = M_{T_{0,0},\, \cN_i}^{Do}(z) + \big(z^2+1\big)K_{0,\beta}(z)^{-1} ( u_1(\ol{z},\dott),\,\cdot\,)_{L^2_r((a,b))} P_{\cN_i}u_1(z,\dott)\big|_{\cN_i}.
\end{equation}
Using \eqref{5.47} with $k=1$ in \eqref{5.55}, one obtains
\begin{align}
M_{T_{0,\b},\, \cN_i}^{Do}(z) &= M_{T_{0,0},\, \cN_i}^{Do}(z)\lb{5.57}\\
&\quad + (i-z)K_{0,\beta}(z)^{-1} ( u_1(\ol{z},\dott),\,\cdot\,)_{L^2_r((a,b))} \sum_{\ell=1}^2 W_{\ell,1}^{Kr}(z)v_{\ell}(i,\dott)\big|_{\cN_i}.\no
\end{align}
Hence, \eqref{5.56} follows from \eqref{5.57} by applying the precise form for $K_{0,\beta}(z)$ given in \eqref{4.56}.  This completes the proof of item $(ii)$.

To prove item $(iii)$, let $\varphi \in [0,2\pi)$ and $R\in \SL(2,\bbR)$ with $R_{1,2}=0$.  By \eqref{5.3}, $M_{T_{\varphi,R},\, \cN_i}^{Do}(\pm i) = \pm iI_{\cN_i}$.  In order to establish \eqref{5.59}, let $z\in \bbC\backslash\bbR$, $z\neq \pm i$, be fixed.  Taking $T_{A,B}=T_{\varphi,R}$ in \eqref{5.2} and invoking \eqref{4.65}, one obtains
\begin{align}\lb{5.58}
M_{T_{\varphi, R},\, \cN_i}^{Do}(z) &= M_{T_{0,0},\, \cN_i}^{Do}(z) \\
&\quad+ \big(z^2+1\big)k_{\varphi,R}(z)^{-1} (u_{\varphi,R}(\ol{z},\dott),\dott)_{L^2_r((a,b))}P_{\cN_i}u_{\varphi,R}(z,\dott)\big|_{\cN_i}.   \no
\end{align}
By \eqref{5.41} and \eqref{5.42},
\begin{align}
\begin{split} 
P_{\cN_i}u_{\varphi,R}(z,\dott) &= \sum_{\ell=1}^2 \big(v_{\ell}(i,\dott),e^{-i\varphi}R_{2,2}u_2(z,\dott)+u_1(z,\dott)\big)_{L_r^2((a,b))}v_{\ell}(i,\dott)\lb{5.60}\\
&=-\frac{1}{z+i}\sum_{\ell=1}^2\big[e^{-i\varphi}R_{2,2}W_{\ell,2}^{Kr}(z)+W_{\ell,1}^{Kr}(z) \big]v_{\ell}(i,\dott). 
\end{split} 
\end{align}
Finally, \eqref{5.59} follows by combining \eqref{5.58} and \eqref{5.60}.
\end{proof}

\section{A Generalized Bessel-Type Operator Example} \lb{s7}

As an illustration of these results, we consider the following explicitly solvable generalized Bessel-type equation following the analysis in \cite{GNS21} (see also \cite{FGKLNS21}).
Let $a=0$, $b\in(0,\infty)\cup\{\infty\}$, and consider
\begin{equation}
\begin{split} \lb{7.1}
p(x)=x^\nu ,\quad r(x)=x^\d, \quad q (x) = \frac{(2+\d-\nu)^2\g^2-(1-\nu)^2}{4}x^{\nu-2}, \\
\d>-1,\ \nu<1,\ \g\geq0,\ x\in(0,b).     
\end{split}
\end{equation}
Then 
\begin{align}
\begin{split}
\tau_{\d,\nu,\g} = x^{-\d}\left[-\frac{d}{dx}x^\nu\frac{d}{dx} +\frac{(2+\d-\nu)^2\g^2-(1-\nu)^2}{4}x^{\nu-2}\right],\\
\d>-1,\; \nu<1,\; \g\geq0,\; x\in(0,b),     \lb{7.2} 
\end{split}
\end{align}
is singular at the endpoint $x=0$ (since the potential, $q$ is not integrable near $x=0$), regular at $x=b$ when $b\in(0,\infty)$, and in the limit point case at $x=b$ when $b=\infty$. Furthermore, $\tau_{\d,\nu,\g}$ is in the limit circle case at $x=0$ if $0\leq\g<1$ and in the limit point case at $x=0$ when $\g\geq1$. 

Solutions to $\tau_{\d,\nu,\g} u=zu$ are given by (cf.\ \cite{KW72}, \cite[No.~2.162, p.~440]{Ka61})
\begin{align}
y_{1,\d,\nu,\g}(z,x)&=x^{(1-\nu)/2} J_{\gamma}\big(2z^{1/2} x^{(2+\d-\nu)/2}/(2+\d-\nu)\big),\quad \g\geq0,\\[1mm] 
y_{2,\d,\nu,\g}(z,x)&=\begin{cases}
x^{(1-\nu)/2} J_{-\gamma}\big(2z^{1/2} x^{(2+\d-\nu)/2}/(2+\d-\nu)\big), & \g\notin\bbN_0,\\
x^{(1-\nu)/2} Y_{\g}\big(2z^{1/2} x^{(2+\d-\nu)/2}/(2+\d-\nu)\big), & \g\in\bbN_0,
\end{cases}\ \g\geq0,
\end{align}
where $J_{\mu}(\dott), Y_{\mu}(\dott)$ are the standard Bessel functions of order $\mu \in \bbR$ 
(cf.\ \cite[Ch.~9]{AS72}).

In the following we assume that 
\begin{equation}
\gamma \in [0,1) 
\end{equation}
to ensure the limit circle case at $x=0$. In this case it suffices to focus on  the generalized boundary values at the singular endpoint $x = 0$ following \cite{GLN20}. For this purpose we introduce principal and nonprincipal solutions $u_{0,\d,\nu,\g}(0, \dott)$ and $\hatt u_{0,\d,\nu,\g}(0, \dott)$ of $\tau_{\d,\nu,\g} u = 0$ at $x=0$ by
\begin{align}
\begin{split} 
u_{0,\d,\nu,\g}(0, x) &= (1-\nu)^{-1}x^{[1-\nu+(2+\d-\nu)\g]/2}, \quad \gamma \in [0,1),   \\
\hatt u_{0,\d,\nu,\g}(0, x) &= \begin{cases} (1-\nu)[(2+\d-\nu) \gamma]^{-1} x^{[1-\nu-(2+\d-\nu)\g]/2}, & \gamma \in (0,1),     \lb{7.6} \\
(1-\nu)x^{(1-\nu)/2} \ln(1/x), & \gamma =0,  \end{cases}\\
&\hspace*{4.5cm} \d>-1,\; \nu<1,\; x \in (0,1).
\end{split} 
\end{align}

\begin{remark}
Since the singularity of $q$ at $x=0$ renders $\tau_{\d,\nu,\g}$ singular at $x=0$ (unless, of 
course, $\gamma = (1-\nu)/(2+\d-\nu)$, in which case $\tau_{\d,\nu,(1-\nu)/(2+\d-\nu)}$ is regular at $x=0$), there is a certain freedom in the choice of the multiplicative constant in the principal solution $u_{0,\d,\nu,\g}$ of $\tau_{\d,\nu,\g} u = 0$ at $x=0$. Our choice of $(1-\nu)^{-1}$ in \eqref{7.6} reflects continuity in the parameters when comparing to boundary conditions in the regular case (cf.\ \cite[Remark~3.12\,$(ii)$]{GLN20}), that is, in the case $\d > -1$, $\nu < 1$, and 
$\gamma = (1-\nu)/(2+\d-\nu)$ treated in \cite{EZ78}. 
\hfill $\diamond$ 
\end{remark}

The generalized boundary values for $g \in \dom(T_{max,\d,\nu,\g})$ are then of the form
\begin{align}
\begin{split} 
\wti g(0) &= - W(u_{0,\d,\nu,\g}(0, \dott), g)(0)   \\
&= \begin{cases} \lim_{x \downarrow 0} g(x)\big/\big[(1-\nu)[(2+\d-\nu) \gamma]^{-1} x^{[1-\nu-(2+\d-\nu)\g]/2}\big], & 
\gamma \in (0,1), \\[1mm]
\lim_{x \downarrow 0} g(x)\big/\big[(1-\nu)x^{(1-\nu)/2} \ln(1/x)\big], & \gamma =0, 
\end{cases} 
\end{split} \\
\begin{split} 
\wti g^{\, \prime} (0) &= W(\hatt u_{0,\d,\nu,\g}(0, \dott), g)(0)   \\
&= \begin{cases} \lim_{x \downarrow 0} \big[g(x) - \wti g(0) (1-\nu)[(2+\d-\nu) \gamma]^{-1} x^{[1-\nu-(2+\d-\nu)\g]/2}\big]\\
\qquad\qquad\big/\big[(1-\nu)^{-1}x^{[1-\nu+(2+\d-\nu)\g]/2}\big], 
& \hspace{-.2cm}\gamma \in (0,1), \\[1mm]
\lim_{x \downarrow 0} \big[g(x) - \wti g(0) (1-\nu)x^{(1-\nu)/2} \ln(1/x)\big]\\
\qquad\qquad\big/\big[(1-\nu)^{-1}x^{(1-\nu)/2}\big], & \hspace{-.2cm} \gamma =0.
\end{cases}
\end{split} 
\end{align}

Next, introducing the standard normalized (at $x=0$) fundamental system of solutions 
$\phi_{\d,\nu,\g}(z, \dott,0), \theta_{\d,\nu,\g}(z, \dott,0)$ of 
$\tau_{\d,\nu,\g} u = z u$, $z \in \bbC$, that is real-valued for $z \in \bbR$ and entire with respect to $z \in \bbC$ by 
\begin{align}
\begin{split} 
\wti \phi_{\d,\nu,\g} (z,0,0) &= 0, \quad 
\wti \phi_{\d,\nu,\g}^{\, \prime} (z,0,0) = 1,    \\ 
\wti \theta_{\d,\nu,\g} (z,0,0) &= 1, \quad \, 
\wti \theta_{\d,\nu,\g}^{\, \prime} (z,0,0) = 0, \quad 
z \in \bbC,     \lb{7.9}
\end{split} 
\end{align}
one obtains explicitly, 
\begin{align}
& \phi_{\d,\nu,\g}(z,x,0) = (1-\nu)^{-1}(2+\d-\nu)^\gamma\Gamma(1+\g) z^{- \gamma/2} 
y_{1,\d,\nu,\g}(z,x),\no \\
& \hspace*{3.15cm} \d>-1,\; \nu<1,\; \gamma \in [0,1),\; z \in \bbC, \; x \in (0,b),      \\
& \theta_{\d,\nu,\g}(z,x,0) = \begin{cases} (1-\nu)(2+\d-\nu)^{-\gamma - 1} \gamma^{-1} \Gamma(1 - \gamma) 
z^{\gamma/2} y_{2,\d,\nu,\g}(z,x), & \hspace{-.1cm} \gamma \in (0,1), \\[1mm]
(1-\nu)(2+\d-\nu)^{-1} [- \pi y_{2,\d,\nu,0}(z,x)\\
\quad+(\ln(z)-2\ln(2+\d-\nu)+2\g_E) y_{1,\d,\nu,0}(z,x)], & \hspace{-.1cm} \gamma =0, 
\end{cases}      \no \\ 
& \hspace*{6.2cm} \d>-1,\; \nu<1,\; z \in \bbC, \; x \in (0,b),    \\
& W(\theta_{\d,\nu,\g}(z,\dott,0), \phi_{\d,\nu,\g}(z,\dott,0)) =1, \quad z \in \bbC,
\end{align}
where $\Gamma(\dott)$ denotes the Gamma function, and $\gamma_{E} = 0.57721\dots$ represents Euler's constant.

We now turn to the cases of computing Donoghue $m$-functions for the generalized Bessel operator in general on the infinite interval and for the Krein--von Neumann extension on the finite interval.

\begin{example}[Infinite Interval] \lb{e7.2} 
Let $b=\infty$. We begin by finding $\psi_{0,\d,\nu,\g}(z,\dott)$ described in Hypothesis \ref{h3.1} for this example.

Since $\tau_{\d,\nu,\g}$ is in the limit point case at $\infty$ $($actually, it is in the strong limit point case at infinity since $q$ is bounded on any interval of the form $[R,\infty)$, $R>0$, and the strong limit point property of $\tau_{\d,\nu,\g =(1-\nu)/(2+\d-\nu)}$ has been shown in \cite{EZ78}$)$, to find the Weyl--Titchmarsh solution and $m$-function corresponding to the Friedrichs $($resp., Dirichlet$)$ boundary condition at $x=0$, one considers the requirement
\begin{align}\lb{7.13}
& \psi_{0,\d,\nu,\g}(z,\dott)=\theta_{\d,\nu,\g}(z,\dott,0) 
+ m_{0,\d,\nu,\g}(z)\phi_{\d,\nu,\g}(z,\dott,0)\in L^2((0,\infty);x^\d dx),    \no \\
& \hspace*{8.85cm} z \in \bbC \backslash [0,\infty). 
\end{align}
This implies
\begin{align}
& \psi_{0,\d,\nu,\g}(z,x) = \begin{cases} i (1-\nu)(2+\d-\nu)^{-\gamma -1} \gamma^{-1} \Gamma(1 - \gamma) \sin(\pi \gamma) 
z^{\gamma/2}  \\
\quad \times x^{(1-\nu)/2} H^{(1)}_{\gamma}\big(2z^{1/2} x^{(2+\d-\nu)/2}/(2+\d-\nu)\big), & \gamma \in (0,1), \\[1mm]
i \pi(1-\nu)/(2+\d-\nu) x^{(1-\nu)/2}\\
\quad \times H^{(1)}_0\big(2z^{1/2} x^{(2+\d-\nu)/2}/(2+\d-\nu)\big), & \gamma = 0, 
\end{cases}    \no \\
& \hspace*{4.95cm} \d>-1,\; \nu<1,\; z \in \bbC \backslash [0,\infty), \; x \in (0,\infty), \label{7.14}   \\
& m_{0,\d,\nu,\g}(z) = \begin{cases} - e^{-i \pi \gamma} (1-\nu)^2(2+\d-\nu)^{- 2 \gamma - 1} \gamma^{-1} \\
\quad \times [\Gamma(1 - \gamma)/\Gamma(1+\gamma)] z^{\gamma}, & \gamma \in (0,1), \\[1mm]
 (1-\nu)^2/(2+\d-\nu)   \\
\quad \times [i \pi -\ln(z)+ 2\ln(2+\d-\nu)- 2\gamma_{E}], & \gamma = 0,    
\end{cases}     \label{7.15} \\
& \hspace*{5.25cm} \d>-1,\; \nu<1,\; z \in \bbC \backslash [0,\infty),    \no 
\end{align}
where $H_{\mu}^{(1)}(\dott)$ is the Hankel function of the first kind and of order $\mu \in \bbR$ 
$($cf.\ \cite[Ch.~9]{AS72}$)$. In particular, it is immediate from \eqref{7.13} and \eqref{7.9} that $\wti \psi_{0,\d,\nu,\g}(z,0)=1$. We mention that the results \eqref{7.14} and \eqref{7.15} coincide with the ones obtained in \cite{GLN20} when $\d=\nu=0$ and \cite{EZ78} when $\g=(1-\nu)/(2+\d-\nu)$.

Substituting the explicit form of $\psi_{0,\d,\nu,\g}(z,\dott)$ given in \eqref{7.14} into Theorems \ref{t6.1} and \ref{t6.2} yields the Friedrichs extension Donoghue $m$-function, $M_{T_{0,\d,\nu,\g},\, \cN_i}^{Do}(z),$ and the Donoghue $m$-function for all other self-adjoint extensions, $M_{T_{\a,\d,\nu,\g},\, \cN_i}^{Do}(z),$ $\a\in(0,\pi),$ respectively. In particular, since $\wti \psi_{0,\d,\nu,\g}^{\, \prime}(z,0)=m_{0,\d,\nu,\g}(z)$ one finds from Theorem \ref{t6.1} and \eqref{7.15},
\begin{align}
M_{T_{0,\d,\nu,\g},\, \cN_i}^{Do}(z)&=\Bigg[-i+\dfrac{m_{0,\d,\nu,\g}(z)-m_{0,\d,\nu,\g}(-i)}{\Im(m_{0,\d,\nu,\g}(i))}\Bigg]I_{\cN_i}  \no \\
&=\begin{cases}\big\{-i-[\sin(\pi\g/2)]^{-1}e^{-i\pi\g}\big(z^\g-e^{3i\pi/2}\big)\big\}I_{\cN_i},  &  \g\in(0,1),\\[1mm]
\{-i+(2/\pi)[(3i\pi/2)-\ln(z)]\}I_{\cN_i},  &  \g=0,
\end{cases}\no\\
&\hspace{4.45cm} \d>-1,\; \nu<1,\; z\in \bbC\backslash[0,\infty),   \lb{7.16}
\end{align}
where the branch of the logarithm is chosen so that $\ln(-i)=3i\pi/2$.  Thus, by Theorem \ref{t6.2} with $\a\in(0,\pi)$,
\begin{align}
M_{T_{\a,\d,\nu,\g},\, \cN_i}^{Do}(z) &= M_{T_{0,\d,\nu,\g},\, \cN_i}^{Do}(z)+ (i-z)\frac{m_{0,\d,\nu,\g}(z)-m_{0,\d,\nu,\g}(-i)}{\cot(\alpha) + m_{0,\d,\nu,\g}(z)}   \no \\
&\quad\; \times (\psi_{0,\d,\nu,\g}(\overline{z},\dott),\dott)_{L_r^2((a,b))}\psi_{0,\d,\nu,\g}(i,\dott)\big|_{\cN_i},  \lb{7.17} \\
&\hspace{1.55cm}\d>-1,\; \nu<1,\; \g\in[0,1),\; z\in\bbC\backslash\bbR.  \no
\end{align}
\end{example}

\begin{example}[Finite Interval] \lb{e7.3} 
Let $b\in(0,\infty)$. It is well known that $T_{min,\d,\nu,\g} \geq \varepsilon I_{L^2_r((a,b))}$ for some $\varepsilon>0$ $($see, e.g., the simpler case $\d=\nu=0$ treated in \cite[Thm. 5.1]{GPS21}$)$. Thus, the Krein--von Neumann extension $T_{0,R_K,\d,\nu,\g}$ of $T_{min,\d,\nu,\g}$ is of the form $($see \cite[Example 4.1]{FGKLNS21}$)$ 
\begin{align}
& T_{0,R_K,\d,\nu,\g} f = \tau_{\d,\nu,\g} f,    \\
& f \in \dom(T_{0,R_K,\d,\nu,\g})=\bigg\{g\in\dom(T_{max,\d,\nu,\g}) \, \bigg| \begin{pmatrix} g(b) 
\\ g^{[1]}(b) \end{pmatrix} = R_{K,\d,\nu,\g} \begin{pmatrix}
\wti g(0) \\ {\wti g}^{\, \prime}(0) \end{pmatrix} \bigg\},     \no
\end{align}
where 
\begin{align}
&R_{K,\d,\nu,\g}=\begin{cases}
b^{[\nu-1-(2+\d-\nu)\g]/2}\\
\quad \times\begin{pmatrix}  \dfrac{1-\nu}{(2+\d-\nu)\g}b^{1-\nu} & \dfrac{1}{1-\nu}b^{1-\nu+(2+\d-\nu)\g} \\
\dfrac{(1-\nu)^2}{2(2+\d-\nu)\g}-\dfrac{1-\nu}{2} & \left[\dfrac{1}{2}+\dfrac{(2+\d-\nu)\g}{2(1-\nu)}\right] b^{(2+\d-\nu)\g}
\end{pmatrix},     \\
\hfill \g\in(0,1), \\[1mm]
\begin{pmatrix}  (1-\nu)\ln(1/b)b^{(1-\nu)/2} & \dfrac{1}{1-\nu}b^{(1-\nu)/2} \\
\dfrac{(1-\nu)^2\ln(1/b)-2(1-\nu)}{2}b^{(\nu-1)/2} & \dfrac{1}{2}b^{(\nu-1)/2}
\end{pmatrix},
\hfill \g=0,
\end{cases} \no \\
&\hspace{9cm}\d>-1,\; \nu<1.
\end{align}

One now explicitly finds the solutions in \eqref{4.52} for this example by choosing 
\begin{align}\lb{7.20}
u_{1,\d,\nu,\g}(z,x)&=\phi_{\d,\nu,\g}(z, x,0)/\phi_{\d,\nu,\g}(z, b,0), \no \\
u_{2,\d,\nu,\g}(z,x)&=\theta_{\d,\nu,\g}(z, x,0)-[\theta_{\d,\nu,\g}(z, b,0)/\phi_{\d,\nu,\g}(z, b,0)]\phi_{\d,\nu,\g}(z, x,0), \\
&\hspace*{3.5cm} \d>-1,\; \nu<1,\; \g\in[0,1),\; x \in (0,b), \no
\end{align}
from which substituting \eqref{7.20} into \eqref{5.27} yields the expressions for $v_{j,\d,\nu,\g}(z,\dott),$ $j=1,2.$ Finally, substituting the expressions for $u_{j,\d,\nu,\g}(z,\dott),$ $v_{j,\d,\nu,\g}(z,\dott),$ $j=1,2,$ and the explicit form of $K_{0,R_K,\d,\nu,\g}(z)$ given in \eqref{4.25a} $($utilizing \eqref{7.6}$)$ into Theorems \ref{t5.1} and \ref{t5.2} yields the Friedrichs extension Donoghue $m$-function, $M_{T_{0,0,\d,\nu,\g},\, \cN_i}^{Do}(z),$ and the Krein--von Neumann extension Donoghue $m$-function, $M_{T_{0,R_K,\d,\nu,\g},\, \cN_i}^{Do}(z),$ respectively.
\end{example}

\medskip

\noindent 
{\bf Acknowledgments.} R.~N. would like to thank the U.S. National Science Foundation for summer support received under Grant DMS-1852288 in connection with {\it REU Site: Research Training for Undergraduates in Mathematical Analysis with Applications in Allied Fields.} M.~P. was supported by the Austrian Science Fund under Grant W1245.

 
\end{document}